\pdfoutput=1  

\documentclass[12pt,a4paper]{article}

\usepackage{amsfonts,amssymb,amsthm,amsmath,latexsym}
\usepackage{eqsection,indent}
\usepackage{subeqnarray, eepicemu, url,cite,bm,mathtools}
\usepackage{multirow, url}  
\usepackage{graphicx,graphbox}
\usepackage{tikz,xcolor}
\usepackage{float}  

\usepackage{forest}

\usepackage[scr=dutchcal,scrscaled=1.05]{mathalfa}

\usepackage{color}

\date{October 27, 2026}

\begin{document}

\title{Generalised hyperbolic polynomials: combinatorial interpretation and real-rootedness}

\author{ 
      \hspace*{-1cm}
      {\large Bishal Deb${}^1$}
   \\[5mm]
     \hspace*{-1.3cm}
     \normalsize
     ${}^1$Yau Mathematical Sciences Center, Tsinghua University, Beijing 100084, China\\
     \normalsize
     \texttt{bishal@gonitsora.com}
}
\maketitle
\thispagestyle{empty}   

\begin{abstract}
For $p,q$ natural numbers, define $\Delta^{\langle p,q \rangle}:=x^{p-1}\frac{\partial}{\partial y} + y^{q-1}\frac{\partial}{\partial x}$,
and consider the sequence of polynomials $\biggl((\Delta^{\langle p,q \rangle})^n \cdot x\biggr)_{n\geq 0}$;
we call these the generalised hyperbolic polynomials.
In this paper, we provide a simple combinatorial interpretation for these polynomials.
As special cases, we obtain combinatorial interpretations for the coefficients of the  analogs of sinh and cosh
corresponding to the  curve $x^p- p y^q/q=1$;
on setting $p=q=3$, we obtain Conrad and Flajolet's tree interpretation for the coefficients of the 
Dixonian elliptic functions (2006) as a special case.
Finally, we show that for any $p,q\geq 3$, after suitably normalising followed by substituting $y=1$,
we obtain a sequence of univariate polynomials which are real-rooted.
\end{abstract}

\bigskip
\noindent
{\bf Key Words:}
Generalised hyperbolic polynomials, arity-parity trees, real-rootedness

\bigskip
\bigskip
\noindent
{\bf Mathematics Subject Classification (MSC 2020) codes:}
05A15 (Primary);
05A19, 33B10, 33C75 (Secondary)

\vspace*{1cm}

\newtheorem{theorem}{Theorem}[section]
\newtheorem{proposition}[theorem]{Proposition}
\newtheorem{lemma}[theorem]{Lemma}
\newtheorem{corollary}[theorem]{Corollary}
\newtheorem{definition}[theorem]{Definition}
\newtheorem{conjecture}[theorem]{Conjecture}
\newtheorem{question}[theorem]{Question}
\newtheorem{problem}[theorem]{Problem}
\newtheorem{openproblem}[theorem]{Open Problem}
\newtheorem{example}[theorem]{Example}

\renewcommand{\theenumi}{\alph{enumi}}
\renewcommand{\labelenumi}{(\theenumi)}
\def\eop{\hbox{\kern1pt\vrule height6pt width4pt
depth1pt\kern1pt}\medskip}
\def\prf{\par\noindent{\bf Proof.\enspace}\rm}
\def\rmk{\par\medskip\noindent{\bf Remark\enspace}\rm}

\newcommand{\textbfit}[1]{\textbf{\textit{#1}}}

\newcommand{\bigdash}{%
\smallskip\begin{center} \rule{5cm}{0.1mm} \end{center}\smallskip}

\newcommand{\safepar}{ {\protect\hfill\protect\break\hspace*{5mm}} }

\newcommand{\be}{\begin{equation}}
\newcommand{\ee}{\end{equation}}
\newcommand{\<}{\langle}
\renewcommand{\>}{\rangle}
\newcommand{\widebar}{\overline}
\def\reff#1{(\protect\ref{#1})}
\def\spose#1{\hbox to 0pt{#1\hss}}
\def\ltapprox{\mathrel{\spose{\lower 3pt\hbox{$\mathchar"218$}}
    \raise 2.0pt\hbox{$\mathchar"13C$}}}
\def\gtapprox{\mathrel{\spose{\lower 3pt\hbox{$\mathchar"218$}}
    \raise 2.0pt\hbox{$\mathchar"13E$}}}
\def\textprime{${}^\prime$}
\def\proof{\par\medskip\noindent{\sc Proof.\ }}
\def\firstproof{\par\medskip\noindent{\sc First Proof.\ }}
\def\secondproof{\par\medskip\noindent{\sc Second Proof.\ }}
\def\alternateproof{\par\medskip\noindent{\sc Alternate Proof.\ }}
\def\algebraicproof{\par\medskip\noindent{\sc Algebraic Proof.\ }}
\def\combinatorialproof{\par\medskip\noindent{\sc Combinatorial Proof.\ }}
\def\proofof#1{\bigskip\noindent{\sc Proof of #1.\ }}
\def\firstproofof#1{\bigskip\noindent{\sc First Proof of #1.\ }}
\def\secondproofof#1{\bigskip\noindent{\sc Second Proof of #1.\ }}
\def\thirdproofof#1{\bigskip\noindent{\sc Third Proof of #1.\ }}
\def\algebraicproofof#1{\bigskip\noindent{\sc Algebraic Proof of #1.\ }}
\def\combinatorialproofof#1{\bigskip\noindent{\sc Combinatorial Proof of #1.\ }}
\def\sketchofproof{\par\medskip\noindent{\sc Sketch of proof.\ }}
\renewcommand{\qed}{ $\square$ \bigskip}
\newcommand{\myendremark}{ $\blacksquare$ \bigskip}
\def\half{ {1 \over 2} }
\def\third{ {1 \over 3} }
\def\twothird{ {2 \over 3} }
\def\smfrac#1#2{{\textstyle{#1\over #2}}}
\def\smhalf{ {\smfrac{1}{2}} }
\newcommand{\real}{\mathop{\rm Re}\nolimits}
\renewcommand{\Re}{\mathop{\rm Re}\nolimits}
\newcommand{\imag}{\mathop{\rm Im}\nolimits}
\renewcommand{\Im}{\mathop{\rm Im}\nolimits}
\newcommand{\sgn}{\mathop{\rm sgn}\nolimits}
\newcommand{\tr}{\mathop{\rm tr}\nolimits}
\newcommand{\supp}{\mathop{\rm supp}\nolimits}
\newcommand{\disc}{\mathop{\rm disc}\nolimits}
\newcommand{\diag}{\mathop{\rm diag}\nolimits}
\newcommand{\tridiag}{\mathop{\rm tridiag}\nolimits}
\newcommand{\AZ}{\mathop{\rm AZ}\nolimits}
\newcommand{\NC}{\mathop{\rm NC}\nolimits}
\newcommand{\PF}{{\rm PF}}
\newcommand{\rk}{\mathop{\rm rk}\nolimits}
\newcommand{\perm}{\mathop{\rm perm}\nolimits}
\def\hboxscript#1{ {\hbox{\scriptsize\em #1}} }
\renewcommand{\emptyset}{\varnothing}
\newcommand{\eqdef}{\stackrel{\rm def}{=}}

\newcommand{\restrict}{\upharpoonright}

\newcommand{\compinv}{{\langle -1 \rangle}}   

\newcommand{\scra}{{\mathcal{A}}}
\newcommand{\scrb}{{\mathcal{B}}}
\newcommand{\scrc}{{\mathcal{C}}}
\newcommand{\bfscra}{{\bm{\mathcal{A}}}}
\newcommand{\bfscrb}{{\bm{\mathcal{B}}}}
\newcommand{\bfscrc}{{\bm{\mathcal{C}}}}
\newcommand{\bfscrap}{{\bm{\mathcal{A}'}}}
\newcommand{\bfscrbp}{{\bm{\mathcal{B}'}}}
\newcommand{\bfscrcp}{{\bm{\mathcal{C}'}}}
\newcommand{\bfscrapp}{{\bm{\mathcal{A}''}}}
\newcommand{\bfscrbpp}{{\bm{\mathcal{B}''}}}
\newcommand{\bfscrcpp}{{\bm{\mathcal{C}''}}}
\newcommand{\scrd}{{\mathcal{D}}}
\newcommand{\scre}{{\mathcal{E}}}
\newcommand{\scrf}{{\mathcal{F}}}
\newcommand{\scrg}{{\mathcal{G}}}
\newcommand{\scrgg}{{\mathscr{g}}}  
\newcommand{\scrh}{{\mathcal{H}}}
\newcommand{\scri}{{\mathcal{I}}}
\newcommand{\scrj}{{\mathcal{J}}}
\newcommand{\scrk}{{\mathcal{K}}}
\newcommand{\scrl}{{\mathcal{L}}}
\newcommand{\scrm}{{\mathcal{M}}}
\newcommand{\scrn}{{\mathcal{N}}}
\newcommand{\scro}{{\mathcal{O}}}
\newcommand\scroo{
  \mathchoice
    {{\scriptstyle\mathcal{O}}}
    {{\scriptstyle\mathcal{O}}}
    {{\scriptscriptstyle\mathcal{O}}}
    {\scalebox{0.6}{$\scriptscriptstyle\mathcal{O}$}}
  }
\newcommand{\scrp}{{\mathcal{P}}}
\newcommand{\scrq}{{\mathcal{Q}}}
\newcommand{\scrr}{{\mathcal{R}}}
\newcommand{\scrs}{{\mathcal{S}}}
\newcommand{\scrss}{{\mathscr{s}}}  
\newcommand{\scrt}{{\mathcal{T}}}
\newcommand{\scrx}{{\mathcal{X}}}
\newcommand{\scry}{{\mathcal{Y}}}
\newcommand{\scrv}{{\mathcal{V}}}
\newcommand{\scrw}{{\mathcal{W}}}
\newcommand{\scrz}{{\mathcal{Z}}}
\newcommand{\SP}{{\mathcal{SP}}}
\newcommand{\ST}{{\mathcal{ST}}}

\newcommand{\bfa}{{\mathbf{a}}}
\newcommand{\bfb}{{\mathbf{b}}}
\newcommand{\bfc}{{\mathbf{c}}}
\newcommand{\bfd}{{\mathbf{d}}}
\newcommand{\bfe}{{\mathbf{e}}}
\newcommand{\bfh}{{\mathbf{h}}}
\newcommand{\bfj}{{\mathbf{j}}}
\newcommand{\bfi}{{\mathbf{i}}}
\newcommand{\bfk}{{\mathbf{k}}}
\newcommand{\bfl}{{\mathbf{l}}}
\newcommand{\bfL}{{\mathbf{L}}}
\newcommand{\bfm}{{\mathbf{m}}}
\newcommand{\bfn}{{\mathbf{n}}}
\newcommand{\bfp}{{\mathbf{p}}}
\newcommand{\bfr}{{\mathbf{r}}}
\newcommand{\bfu}{{\mathbf{u}}}
\newcommand{\bfv}{{\mathbf{v}}}
\newcommand{\bfw}{{\mathbf{w}}}
\newcommand{\bfx}{{\mathbf{x}}}
\newcommand{\bfy}{{\mathbf{y}}}
\newcommand{\bfz}{{\mathbf{z}}}
\renewcommand{\k}{{\mathbf{k}}}
\newcommand{\n}{{\mathbf{n}}}
\newcommand{\vv}{{\mathbf{v}}}
\newcommand{\bv}{{\mathbf{v}}}
\newcommand{\w}{{\mathbf{w}}}
\newcommand{\x}{{\mathbf{x}}}
\newcommand{\y}{{\mathbf{y}}}
\newcommand{\cc}{{\mathbf{c}}}
\newcommand{\zero}{{\mathbf{0}}}
\newcommand{\one}{{\mathbf{1}}}
\newcommand{\bmm}{{\mathbf{m}}}

\newcommand{\ahat}{{\widehat{a}}}
\newcommand{\vhat}{{\widehat{v}}}
\newcommand{\yhat}{{\widehat{y}}}
\newcommand{\phat}{{\widehat{p}}}
\newcommand{\qhat}{{\widehat{q}}}
\newcommand{\Xhat}{{\widehat{X}}}
\newcommand{\Zhat}{{\widehat{Z}}}
\newcommand{\myhat}{{\widehat{\;}}}
\newcommand{\vtilde}{{\widetilde{v}}}
\newcommand{\ytilde}{{\widetilde{y}}}

\newcommand{\C}{{\mathbb C}}
\newcommand{\D}{{\mathbb D}}
\newcommand{\Z}{{\mathbb Z}}
\newcommand{\N}{{\mathbb N}}
\newcommand{\Q}{{\mathbb Q}}
\newcommand{\PP}{{\mathbb P}}
\newcommand{\R}{{\mathbb R}}
\newcommand{\RR}{{\mathbb R}}
\newcommand{\E}{{\mathbb E}}

\newcommand{\Sym}{{\mathfrak{S}}}
\newcommand{\SymB}{{\mathfrak{B}}}
\newcommand{\Alt}{{\mathrm{Alt}}}
\newcommand{\dperm}{{\mathfrak{D}}}
\newcommand{\dcycle}{{\mathfrak{DC}}}

\newcommand{\germanA}{{\mathfrak{A}}}
\newcommand{\germanB}{{\mathfrak{B}}}
\newcommand{\germanQ}{{\mathfrak{Q}}}
\newcommand{\germanh}{{\mathfrak{h}}}

\newcommand{\myle}{\preceq}
\newcommand{\myge}{\succeq}
\newcommand{\mygt}{\succ}

\newcommand{\B}{{\sf B}}
\newcommand{\OB}{B^{\rm ord}}
\newcommand{\OS}{{\sf OS}}
\newcommand{\OO}{{\sf O}}
\newcommand{\OSP}{{\sf OSP}}
\newcommand{\Eu}{{\sf Eu}}
\newcommand{\ERR}{{\sf ERR}}
\newcommand{\sfB}{{\sf B}}
\newcommand{\sfD}{{\sf D}}
\newcommand{\sfE}{{\sf E}}
\newcommand{\sfG}{{\sf G}}
\newcommand{\sfJ}{{\sf J}}
\newcommand{\sfP}{{\sf P}}
\newcommand{\sfQ}{{\sf Q}}
\newcommand{\sfS}{{\sf S}}
\newcommand{\sfT}{{\sf T}}
\newcommand{\sfW}{{\sf W}}
\newcommand{\sfMV}{{\sf MV}}
\newcommand{\AMV}{{\sf AMV}}
\newcommand{\BM}{{\sf BM}}
\newcommand{\emIB}{B^{\rm irr}}
\newcommand{\emIP}{P^{\rm irr}}
\newcommand{\emOB}{B^{\rm ord}}
\newcommand{\emCB}{B^{\rm cyc}}
\newcommand{\emSC}{P^{\rm cyc}}

\newcommand{\lev}{{\rm lev}}
\newcommand{\stat}{{\rm stat}}
\newcommand{\cyc}{{\rm cyc}}
\newcommand{\mysteryone}{{\rm mys1}}
\newcommand{\mysterytwo}{{\rm mys2}}
\newcommand{\Asc}{{\rm Asc}}
\newcommand{\asc}{{\rm asc}}
\newcommand{\Des}{{\rm Des}}
\newcommand{\des}{{\rm des}}
\newcommand{\Exc}{{\rm Exc}}

\newcommand{\EArec}{{\rm EArec}}
\newcommand{\earec}{{\rm earec}}
\newcommand{\recarec}{{\rm recarec}}
\newcommand{\erec}{{\rm erec}}
\newcommand{\nonrec}{{\rm nonrec}}
\newcommand{\nrar}{{\rm nrar}}
\newcommand{\ereccval}{{\rm ereccval}}
\newcommand{\ereccdrise}{{\rm ereccdrise}}
\newcommand{\eareccpeak}{{\rm eareccpeak}}
\newcommand{\eareccdfall}{{\rm eareccdfall}}
\newcommand{\eareccval}{{\rm eareccval}}
\newcommand{\ereccpeak}{{\rm ereccpeak}}
\newcommand{\rar}{{\rm rar}}
\newcommand{\evenrar}{{\rm evenrar}}
\newcommand{\oddrar}{{\rm oddrar}}
\newcommand{\nrcpeak}{{\rm nrcpeak}}
\newcommand{\nrcval}{{\rm nrcval}}
\newcommand{\nrcdrise}{{\rm nrcdrise}}
\newcommand{\nrcdfall}{{\rm nrcdfall}}
\newcommand{\nrfix}{{\rm nrfix}}
\newcommand{\Evenfix}{{\rm Evenfix}}
\newcommand{\evenfix}{{\rm evenfix}}
\newcommand{\Oddfix}{{\rm Oddfix}}
\newcommand{\oddfix}{{\rm oddfix}}
\newcommand{\evennrfix}{{\rm evennrfix}}
\newcommand{\oddnrfix}{{\rm oddnrfix}}
\newcommand{\Cpeak}{{\rm Cpeak}}
\newcommand{\cpeak}{{\rm cpeak}}
\newcommand{\Cval}{{\rm Cval}}
\newcommand{\cval}{{\rm cval}}
\newcommand{\Cdasc}{{\rm Cdasc}}
\newcommand{\cdasc}{{\rm cdasc}}
\newcommand{\Cddes}{{\rm Cddes}}
\newcommand{\cddes}{{\rm cddes}}
\newcommand{\Cdrise}{{\rm Cdrise}}
\newcommand{\cdrise}{{\rm cdrise}}
\newcommand{\Cdfall}{{\rm Cdfall}}
\newcommand{\cdfall}{{\rm cdfall}}

\newcommand{\maxpeak}{{\rm maxpeak}}
\newcommand{\nmaxpeak}{{\rm nmaxpeak}}
\newcommand{\minval}{{\rm minval}}
\newcommand{\nminval}{{\rm nminval}}

\newcommand{\exc}{{\rm exc}}
\newcommand{\excee}{{\rm excee}}
\newcommand{\exceo}{{\rm exceo}}
\newcommand{\excoe}{{\rm excoe}}
\newcommand{\excoo}{{\rm excoo}}
\newcommand{\erecexcoe}{{\rm erecexcoe}}
\newcommand{\erecexcoo}{{\rm erecexcoo}}
\newcommand{\nrexcoe}{{\rm nrexcoe}}
\newcommand{\nrexcoo}{{\rm nrexcoo}}
\newcommand{\aexc}{{\rm aexc}}
\newcommand{\aexcee}{{\rm aexcee}}
\newcommand{\aexceo}{{\rm aexceo}}
\newcommand{\aexcoe}{{\rm aexcoe}}
\newcommand{\aexcoo}{{\rm aexcoo}}
\newcommand{\earecaexcee}{{\rm earecaexcee}}
\newcommand{\earecaexceo}{{\rm earecaexceo}}
\newcommand{\nraexcee}{{\rm nraexcee}}
\newcommand{\nraexceo}{{\rm nraexceo}}
\newcommand{\Fix}{{\rm Fix}}
\newcommand{\fix}{{\rm fix}}
\newcommand{\fixe}{{\rm fixe}}
\newcommand{\fixo}{{\rm fixo}}
\newcommand{\rare}{{\rm rare}}
\newcommand{\raro}{{\rm raro}}
\newcommand{\nrfixe}{{\rm nrfixe}}
\newcommand{\nrfixo}{{\rm nrfixo}}
\newcommand{\xee}{x_{\rm ee}}
\newcommand{\xeo}{x_{\rm eo}}
\newcommand{\uee}{u_{\rm ee}}
\newcommand{\ueo}{u_{\rm eo}}
\newcommand{\yoo}{y_{\rm oo}}
\newcommand{\yoe}{y_{\rm oe}}
\newcommand{\voo}{v_{\rm oo}}
\newcommand{\voe}{v_{\rm oe}}
\newcommand{\zo}{z_{\rm o}}
\newcommand{\ze}{z_{\rm e}}
\newcommand{\wo}{w_{\rm o}}
\newcommand{\we}{w_{\rm e}}
\newcommand{\so}{s_{\rm o}}
\newcommand{\se}{s_{\rm e}}

\newcommand{\Wex}{{\rm Wex}}
\newcommand{\wex}{{\rm wex}}
\newcommand{\lrmax}{{\rm lrmax}}
\newcommand{\rlmax}{{\rm rlmax}}
\newcommand{\Rec}{{\rm Rec}}
\newcommand{\rec}{{\rm rec}}
\newcommand{\Arec}{{\rm Arec}}
\newcommand{\arec}{{\rm arec}}
\newcommand{\ERec}{{\rm ERec}}
\newcommand{\Val}{{\rm Val}}
\newcommand{\val}{{\rm val}}
\newcommand{\dasc}{{\rm dasc}}
\newcommand{\ddes}{{\rm ddes}}
\newcommand{\inv}{{\rm inv}}
\newcommand{\maj}{{\rm maj}}
\newcommand{\rs}{{\rm rs}}
\newcommand{\cross}{{\rm cr}}
\newcommand{\crosshat}{{\widehat{\rm cr}}}
\newcommand{\nest}{{\rm ne}}
\newcommand{\ucross}{{\rm ucross}}
\newcommand{\ucrosscval}{{\rm ucrosscval}}
\newcommand{\ucrosscpeak}{{\rm ucrosscpeak}}
\newcommand{\ucrosscdrise}{{\rm ucrosscdrise}}
\newcommand{\lcross}{{\rm lcross}}
\newcommand{\lcrosscpeak}{{\rm lcrosscpeak}}
\newcommand{\lcrosscval}{{\rm lcrosscval}}
\newcommand{\lcrosscdfall}{{\rm lcrosscdfall}}
\newcommand{\unest}{{\rm unest}}
\newcommand{\unestcval}{{\rm unestcval}}
\newcommand{\unestcpeak}{{\rm unestcpeak}}
\newcommand{\unestcdrise}{{\rm unestcdrise}}
\newcommand{\lnest}{{\rm lnest}}
\newcommand{\lnestcpeak}{{\rm lnestcpeak}}
\newcommand{\lnestcval}{{\rm lnestcval}}
\newcommand{\lnestcdfall}{{\rm lnestcdfall}}
\newcommand{\ujoin}{{\rm ujoin}}
\newcommand{\ljoin}{{\rm ljoin}}
\newcommand{\psnest}{{\rm psnest}}
\newcommand{\upsnest}{{\rm upsnest}}
\newcommand{\lpsnest}{{\rm lpsnest}}
\newcommand{\epsnest}{{\rm epsnest}}
\newcommand{\opsnest}{{\rm opsnest}}
\newcommand{\rodd}{{\rm rodd}}
\newcommand{\reven}{{\rm reven}}
\newcommand{\lodd}{{\rm lodd}}
\newcommand{\leven}{{\rm leven}}
\newcommand{\sg}{{\rm sg}}
\newcommand{\bl}{{\rm bl}}
\newcommand{\tran}{{\rm tr}}
\newcommand{\area}{{\rm area}}
\newcommand{\ret}{{\rm ret}}
\newcommand{\peaks}{{\rm peaks}}
\newcommand{\hl}{{\rm hl}}
\newcommand{\sll}{{\rm sl}}
\newcommand{\negg}{{\rm neg}}
\newcommand{\imp}{{\rm imp}}
\newcommand{\osg}{{\rm osg}}
\newcommand{\ons}{{\rm ons}}
\newcommand{\isg}{{\rm isg}}
\newcommand{\ins}{{\rm ins}}
\newcommand{\LL}{{\rm LL}}
\newcommand{\height}{{\rm ht}}
\newcommand{\as}{{\rm as}}

\newcommand{\ba}{{\bm{a}}}
\newcommand{\bahat}{{\widehat{\bm{a}}}}
\newcommand{\bb}{{\bm{b}}}
\newcommand{\bc}{{\bm{c}}}
\newcommand{\bchat}{{\widehat{\bm{c}}}}
\newcommand{\bd}{{\bm{d}}}
\newcommand{\bee}{{\bm{e}}}
\newcommand{\beh}{{\bm{eh}}}
\newcommand{\bff}{{\bm{f}}}
\newcommand{\bg}{{\bm{g}}}
\newcommand{\bh}{{\bm{h}}}
\newcommand{\bll}{{\bm{\ell}}}
\newcommand{\bp}{{\bm{p}}}
\newcommand{\br}{{\bm{r}}}
\newcommand{\bs}{{\bm{s}}}
\newcommand{\bu}{{\bm{u}}}
\newcommand{\bw}{{\bm{w}}}
\newcommand{\bx}{{\bm{x}}}
\newcommand{\by}{{\bm{y}}}
\newcommand{\bz}{{\bm{z}}}
\newcommand{\bA}{{\bm{A}}}
\newcommand{\bB}{{\bm{B}}}
\newcommand{\bC}{{\bm{C}}}
\newcommand{\bE}{{\bm{E}}}
\newcommand{\bF}{{\bm{F}}}
\newcommand{\bG}{{\bm{G}}}
\newcommand{\bH}{{\bm{H}}}
\newcommand{\bI}{{\bm{I}}}
\newcommand{\bJ}{{\bm{J}}}
\newcommand{\bM}{{\bm{M}}}
\newcommand{\bN}{{\bm{N}}}
\newcommand{\bP}{{\bm{P}}}
\newcommand{\bQ}{{\bm{Q}}}
\newcommand{\bR}{{\bm{R}}}
\newcommand{\bS}{{\bm{S}}}
\newcommand{\bT}{{\bm{T}}}
\newcommand{\bW}{{\bm{W}}}
\newcommand{\bX}{{\bm{X}}}
\newcommand{\bY}{{\bm{Y}}}
\newcommand{\bIB}{{\bm{B}^{\rm irr}}}
\newcommand{\bOB}{{\bm{B}^{\rm ord}}}
\newcommand{\bOS}{{\bm{OS}}}
\newcommand{\bERR}{{\bm{ERR}}}
\newcommand{\bSP}{{\bm{SP}}}
\newcommand{\bMV}{{\bm{MV}}}
\newcommand{\bBM}{{\bm{BM}}}
\newcommand{\balpha}{{\bm{\alpha}}}
\newcommand{\balphapre}{{\bm{\alpha}^{\rm pre}}}
\newcommand{\bbeta}{{\bm{\beta}}}
\newcommand{\bgamma}{{\bm{\gamma}}}
\newcommand{\bdelta}{{\bm{\delta}}}
\newcommand{\bkappa}{{\bm{\kappa}}}
\newcommand{\bmu}{{\bm{\mu}}}
\newcommand{\bomega}{{\bm{\omega}}}
\newcommand{\bsigma}{{\bm{\sigma}}}
\newcommand{\btau}{{\bm{\tau}}}
\newcommand{\bphi}{{\bm{\phi}}}
\newcommand{\bpsi}{{\bm{\psi}}}
\newcommand{\bzeta}{{\bm{\zeta}}}
\newcommand{\bone}{{\bm{1}}}
\newcommand{\bzero}{{\bm{0}}}

\newcommand{\sfa}{{{\sf a}}}
\newcommand{\sfb}{{{\sf b}}}
\newcommand{\sfc}{{{\sf c}}}
\newcommand{\sfd}{{{\sf d}}}
\newcommand{\sfe}{{{\sf e}}}
\newcommand{\sff}{{{\sf f}}}
\newcommand{\sfg}{{{\sf g}}}
\newcommand{\sfh}{{{\sf h}}}
\newcommand{\sfi}{{{\sf i}}}
\newcommand{\sfX}{{{\sf X}}}
\newcommand{\sfY}{{{\sf Y}}}
\newcommand{\bsfa}{{\mbox{\textsf{\textbf{a}}}}}
\newcommand{\bsfb}{{\mbox{\textsf{\textbf{b}}}}}
\newcommand{\bsfc}{{\mbox{\textsf{\textbf{c}}}}}
\newcommand{\bsfd}{{\mbox{\textsf{\textbf{d}}}}}
\newcommand{\bsfe}{{\mbox{\textsf{\textbf{e}}}}}
\newcommand{\bsff}{{\mbox{\textsf{\textbf{f}}}}}
\newcommand{\bsfg}{{\mbox{\textsf{\textbf{g}}}}}
\newcommand{\bsfh}{{\mbox{\textsf{\textbf{h}}}}}
\newcommand{\bsfi}{{\mbox{\textsf{\textbf{i}}}}}

\newcommand{\Cbar}{{\overline{C}}}
\newcommand{\Dbar}{{\overline{D}}}
\newcommand{\dbar}{{\overline{d}}}
\def\Ctilde{{\widetilde{C}}}
\def\Ftilde{{\widetilde{F}}}
\def\Gtilde{{\widetilde{G}}}
\def\Htilde{{\widetilde{H}}}
\def\Ptilde{{\widetilde{P}}}
\def\Chat{{\widehat{C}}}
\def\ctilde{{\widetilde{c}}}
\def\zbar{{\overline{Z}}}
\def\pitilde{{\widetilde{\pi}}}
\def\omegahat{{\widehat{\omega}}}

\newcommand{\sech}{{\rm sech}}

\newcommand{\sinv}{\sigma^{-1}}

%
%
\newcommand{\sn}{{\rm sn}}
\newcommand{\cn}{{\rm cn}}
\newcommand{\dn}{{\rm dn}}
\newcommand{\sm}{{\rm sm}}
\newcommand{\smh}{{\rm smh}}
\newcommand{\cm}{{\rm cm}}
\newcommand{\cmh}{{\rm cmh}}

\newcommand{\Arg}{{\rm Arg}}

%
%
\newcommand{\zfz}{ {{}_0 \! F_0} }
\newcommand{\zfo}{ {{}_0  F_1} }
\newcommand{\ofz}{ {{}_1 \! F_0} }
\newcommand{\ofo}{ {{}_1 \! F_1} }
\newcommand{\oft}{ {{}_1 \! F_2} }

%
%
\newcommand{\FHyper}[2]{ {\tensor[_{#1 \!}]{F}{_{#2}}\!} }
\newcommand{\FHYPER}[5]{ {\FHyper{#1}{#2} \!\biggl(
   \!\!\begin{array}{c} #3 \\[1mm] #4 \end{array}\! \bigg|\, #5 \! \biggr)} }
\newcommand{\tfo}{ {\FHyper{2}{1}} }
\newcommand{\tfz}{ {\FHyper{2}{0}} }
\newcommand{\threefz}{ {\FHyper{3}{0}} }
\newcommand{\FHYPERbottomzero}[3]{ {\FHyper{#1}{0} \hspace*{-0mm}\biggl(
   \!\!\begin{array}{c} #2 \\[1mm] \hbox{---} \end{array}\! \bigg|\, #3 \! \biggr)} }
\newcommand{\FHYPERtopzero}[3]{ {\FHyper{0}{#1} \hspace*{-0mm}\biggl(
   \!\!\begin{array}{c} \hbox{---} \\[1mm] #2 \end{array}\! \bigg|\, #3 \! \biggr)} }

\newcommand{\phiHyper}[2]{ {\tensor[_{#1}]{\phi}{_{#2}}} }
\newcommand{\psiHyper}[2]{ {\tensor[_{#1}]{\psi}{_{#2}}} }
\newcommand{\PhiHyper}[2]{ {\tensor[_{#1}]{\Phi}{_{#2}}} }
\newcommand{\PsiHyper}[2]{ {\tensor[_{#1}]{\Psi}{_{#2}}} }
\newcommand{\phiHYPER}[6]{ {\phiHyper{#1}{#2} \!\left(
   \!\!\begin{array}{c} #3 \\ #4 \end{array}\! ;\, #5, \, #6 \! \right)\!} }
\newcommand{\psiHYPER}[6]{ {\psiHyper{#1}{#2} \!\left(
   \!\!\begin{array}{c} #3 \\ #4 \end{array}\! ;\, #5, \, #6 \! \right)} }
\newcommand{\PhiHYPER}[5]{ {\PhiHyper{#1}{#2} \!\left(
   \!\!\begin{array}{c} #3 \\ #4 \end{array}\! ;\, #5 \! \right)\!} }
\newcommand{\PsiHYPER}[5]{ {\PsiHyper{#1}{#2} \!\left(
   \!\!\begin{array}{c} #3 \\ #4 \end{array}\! ;\, #5 \! \right)\!} }
\newcommand{\zerophizero}{ {\phiHyper{0}{0}} }
\newcommand{\ophizero}{ {\phiHyper{1}{0}} }
\newcommand{\zphio}{ {\phiHyper{0}{1}} }
\newcommand{\ophio}{ {\phiHyper{1}{1}} }
\newcommand{\tphio}{ {\phiHyper{2}{1}} }
\newcommand{\tphiz}{ {\phiHyper{2}{0}} }
\newcommand{\tPhio}{ {\PhiHyper{2}{1}} }
\newcommand{\opsio}{ {\psiHyper{1}{1}} }

%
%
\newcommand{\stirlingsubset}[2]{\genfrac{\{}{\}}{0pt}{}{#1}{#2}}
\newcommand{\stirlingcycle}[2]{\genfrac{[}{]}{0pt}{}{#1}{#2}}
\newcommand{\assocstirlingsubset}[3]{{\genfrac{\{}{\}}{0pt}{}{#1}{#2}}_{\! \ge #3}}
\newcommand{\genstirlingsubset}[4]{{\genfrac{\{}{\}}{0pt}{}{#1}{#2}}_{\! #3,#4}}
\newcommand{\irredstirlingsubset}[2]{{\genfrac{\{}{\}}{0pt}{}{#1}{#2}}^{\!\rm irr}}
\newcommand{\euler}[2]{\genfrac{\langle}{\rangle}{0pt}{}{#1}{#2}}
\newcommand{\eulergen}[3]{{\genfrac{\langle}{\rangle}{0pt}{}{#1}{#2}}_{\! #3}}
\newcommand{\eulersecond}[2]{\left\langle\!\! \euler{#1}{#2} \!\!\right\rangle}
\newcommand{\eulersecondgen}[3]{{\left\langle\!\! \euler{#1}{#2} \!\!\right\rangle}_{\! #3}}
\newcommand{\binomvert}[2]{\genfrac{\vert}{\vert}{0pt}{}{#1}{#2}}
\newcommand{\binomsquare}[2]{\genfrac{[}{]}{0pt}{}{#1}{#2}}
\newcommand{\doublebinom}[2]{\left(\!\! \binom{#1}{#2} \!\!\right)}


\newenvironment{sarray}{
             \textfont0=\scriptfont0
             \scriptfont0=\scriptscriptfont0
             \textfont1=\scriptfont1
             \scriptfont1=\scriptscriptfont1
             \textfont2=\scriptfont2
             \scriptfont2=\scriptscriptfont2
             \textfont3=\scriptfont3
             \scriptfont3=\scriptscriptfont3
           \renewcommand{\arraystretch}{0.7}
           \begin{array}{l}}{\end{array}}

\newenvironment{scarray}{
             \textfont0=\scriptfont0
             \scriptfont0=\scriptscriptfont0
             \textfont1=\scriptfont1
             \scriptfont1=\scriptscriptfont1
             \textfont2=\scriptfont2
             \scriptfont2=\scriptscriptfont2
             \textfont3=\scriptfont3
             \scriptfont3=\scriptscriptfont3
           \renewcommand{\arraystretch}{0.7}
           \begin{array}{c}}{\end{array}}


\newcommand*\circled[1]{\tikz[baseline=(char.base)]{
  \node[shape=circle,draw,inner sep=1pt] (char) {#1};}}
\newcommand{\ostar}{{\circledast}}
\newcommand{\ostarN}{{\,\circledast_{\vphantom{\dot{N}}N}\,}}
\newcommand{\ostarPsi}{{\,\circledast_{\vphantom{\dot{\Psi}}\Psi}\,}}
\newcommand{\starN}{{\,\ast_{\vphantom{\dot{N}}N}\,}}
\newcommand{\starpsi}{{\,\ast_{\vphantom{\dot{\bpsi}}\!\bpsi}\,}}
\newcommand{\starone}{{\,\ast_{\vphantom{\dot{1}}1}\,}}
\newcommand{\startwo}{{\,\ast_{\vphantom{\dot{2}}2}\,}}
\newcommand{\starinfty}{{\,\ast_{\vphantom{\dot{\infty}}\infty}\,}}
\newcommand{\starT}{{\,\ast_{\vphantom{\dot{T}}T}\,}}

\newcommand*{\Scale}[2][4]{\scalebox{#1}{$#2$}}

\newcommand*{\Scaletext}[2][4]{\scalebox{#1}{#2}} 

\newcommand{\operatorpq}{\Delta^{\langle p,q \rangle}}

\section{Introduction}

For $p,q\in \mathbb{N}_{\geq 1}$ and real numbers $a,b, x_0, y_0 \in \mathbb{R}_{\geq 0}$, consider the following system of first-order ordinary differential equations:
\begin{eqnarray}
x'(t)
&=&
a (y(t))^{q-1}\;,
\qquad 
x(0)  = x_0\;;
\nonumber
\\
y'(t)
&=&
b \,(x(t))^{p-1}\;,
\qquad 
y(0) = y_0\;.
\label{eq.pqODE}
\end{eqnarray}
The solution $(x(t),y(t))$ to this system satisfies
\be
b q \, x(t)^p - a p \, y(t)^q \;=\; bqx_0^p - ap y_0^p;
\label{eq.firstcurve}
\ee
For $a,b, p,q>0$, let $c,d$ be given by
$c = x_0^p - \frac{ap}{bq} y_0^q, d = \frac{bq}{ap} x_0^p - y_0^q$,
and assume that $c,d>0$.
Then we may rewrite~\eqref{eq.firstcurve}
as 
\be
 \dfrac{ x(t)^p}{c} -\dfrac{ y(t)^q}{d} \;=\; 1\;.
\label{eq.curve}
\ee
This curve is sometimes known as the {\em generalised Lam\'e curve}.
When $c=d=1$, $p=q$, the curve $x^p - y^p =1$ is the  hyperbolic analog of the Fermat curve $x^p+y^p=1$.
When $p=q=2$, $x_0=1,y_0=0, a=b=1$, the solution $(x(t),y(t))$ is the pair of hyperbolic functions $(\cosh({t}),\sinh({t}))$.
Therefore, we will call this system~\eqref{eq.pqODE} the 
{\bf $\bm{(p,q)}$-{\em hyperbolic system}}. 

Let us assume $a=b=x_0=1,y_0=0$ in the system~\eqref{eq.pqODE},
and instead of using the symbols `$x$',`$y$' to be functions,
we now switch attention and consider these symbols to be indeterminates and consider the polynomial ring
$\mathbb{C}[x,y]$.
We mimic the $(p,q)$-hyperbolic system by considering instead,
the following differential operator acting on the polynomial ring $\mathbb{C}[x,y]$:
\be
\Delta^{\langle p,q \rangle}
\;\eqdef\; 
x^{p-1} \dfrac{\partial}{\partial y} 
\,+\,  y^{q-1} \dfrac{\partial}{\partial x}
\,.
\ee
Define the polynomial
$X^{\langle p,q \rangle}_{0}(x,y) \coloneqq x$,
and for $n\geq 1$ let 
\be
X^{\langle p,q \rangle}_{n}(x,y) 
\; \eqdef \;
\Delta^{\langle p,q \rangle} \cdot X^{\langle p,q \rangle}_{n-1}(x,y) 
\;=\;
(\Delta^{\langle p,q \rangle})^n \cdot x\;.
\label{eq.def.gentrig}
\ee
When $p,q$ are specified, we will refer to these polynomials as the 
{\bf $\bm{(p,q)}$-{\em hyperbolic polynomials}},
and in general, we refer to these polynomials as the {\bf {\em generalised hyperbolic polynomials}}.

%

The following observations are immediate from the definition~
\eqref{eq.def.gentrig}:
\begin{itemize}
\item Notice that $q=1$ forces $X^{\langle p,1\rangle}_n(x,y) = 0$ for all $n\geq 2$,
and $p=1$ forces $X^{\langle 1,q \rangle}_n(x,y) =0$ for all $n\geq q+1$.
Therefore, in the rest of this paper, we will impose $p,q\geq 2$.

\item It is clear from the definition that the coefficients of 
$X^{\langle p,q \rangle}_{n}(x,y)$
are all non-negative integers.

\item We may also choose to start from $y$ instead of $x$ and define $Y^{\langle p,q\rangle}_0(x,y) \coloneqq y$,
$Y^{\langle p,q\rangle}_n(x,y) \eqdef  (\Delta^{\langle p,q \rangle})^n\cdot y$.
However, we do not gain anything new due to the following identity
\be
Y^{\langle p,q \rangle}_n(x,y) 
\;=\;
X^{\langle q,p \rangle}_n(y,x)\;.
\label{eq.xy}
\ee


\end{itemize}

We list out the first few polynomials $X^{\langle p,q \rangle}_n(x,y)$ for generic $p,q \geq 2$:
\begin{subeqnarray}
	X^{\langle p,q \rangle}_0(x,y) &=& x\\
	X^{\langle p,q \rangle}_1(x,y) &=& y^{q-1}\\
	X^{\langle p,q \rangle}_2(x,y) &=& (q-1)x^{p-1} y^{q-2}\\
	X^{\langle p,q \rangle}_3(x,y) &=& (q-1)x^{p-2} y^{q-3}\left[
		 (q-2) x^p \,+\, (p-1) y^q
		\right]
\label{eq.compute.xpq0123}
\end{subeqnarray}

Our first contribution in this paper is a combinatorial interpretation for the generalised hyperbolic polynomials
in terms of some families of increasing labelled trees;
we do this in Section~\ref{sec.combinter}.

Then in Section~\ref{sec.reduced}, corresponding to each polynomial $X^{\langle p,q \rangle}_n$ we introduce
a bivariate homogeneous polynomial $\Xhat^{\langle p,q \rangle}_n(x,y)$;
we call these the reduced hyperbolic polynomials.
In fact, in Theorem~\ref{prop.realrooted}, we show that these two polynomials are the same up to a substitution $\{x\mapsto x^p, y\mapsto y^q\}$ and multiplication by a simple factor.
The reader may compare the first few polynomials stated in~\eqref{eq.compute.xpq0123} and~\eqref{eq.compute.xhat0123} 
to get a feeling for this.

Our second main contribution is the following theorem.

\begin{theorem}
Let $p,q \geq 3$.
For all $n\geq 3$, the polynomials $\Xhat^{\langle p,q \rangle}_n(x,1)$ are real-rooted.
\label{thm.xhat.realrooted}
\end{theorem}

We will prove this theorem in Section~\ref{sec.realrooted}.

Before proceeding to our main results, we provide some context behind our work in Section~\ref{sec.context}.
Then in Section~\ref{sec.cosh.sinh.Maclaurin},
we demonstrate how the Maclaurin expansions of the solutions to the generalised hyperbolic system
can be obtained from the generalised hyperbolic polynomials.

\subsection{Generalised trigonometric and hyperbolic systems and Schett-like generation}
\label{sec.context}

Let us state our reasons for introducing the generalised hyperbolic polynomials and the context in which they arise.
Two decades ago, Conrad and Flajolet initiated the combinatorial theory associated to the Fermat cubic $x^3+y^3=1$
in their beautiful paper \cite{Conrad_06} (also see Conrad's PhD thesis \cite{Conrad_02}).
The central object of their paper was the pair of {\em Dixonian elliptic functions} $(\cm(t),\sm(t))$ 
which are the solution to the system of differential equations
\be
x' = -y^2 \qquad y' = x^2 \qquad \text{with initial conditions } x(0)=1, y(0) =0\;.
\label{eq.dixon.diff}
\ee
The first detailed study of $\cm$, $\sm$ was carried out by Dixon \cite{Dixon_1890} after whom these functions are named.
In their abstract, Conrad and Flajolet described these functions to be ``pregnant with interesting combinatorics:''
they constructed several combinatorial interpretations for the Maclaurin coefficients (including tree and permutation interpretations), 
stated several interesting continued fractions,
established a connection to P\'olya urns in probability theory (see \cite{Flajolet_05, Flajolet_06} for more on urn models),
and motivated many natural questions in analysis, combinatorics, and probability.
In \cite[Section~3.1]{Conrad_06}, the authors considered the operator $\Delta^{\langle 3,3 \rangle}$
and the associated $(3,3)$-hyperbolic polynomials.
However, their consideration of a differential operator associated to a system of differential equations
was, in fact, motivated by an earlier work of Dumont.
Dumont introduced the {\em Schett polynomials} around 1980 \cite{Dumont_79,Dumont_81}
by using the operator $yz \dfrac{\partial}{\partial x} + xz \dfrac{\partial}{\partial y} + xy \dfrac{\partial}{\partial z}$
and it was related to the system
\be
x' = yz, \quad y' = zx, \quad z'=xy\;,
\ee
which are associated to the Jacobi elliptic functions $\sn,\cn$ and $\dn$.
It must be mentioned that parallel to Dumont's work, 
Flajolet \cite{Flajolet_80}
and  Viennot \cite{Viennot_80} 
also initiated the combinatorial study of Jacobi elliptic functions at the same time.

It is natural to ask if one can further this investigation by considering higher degree generalisations of the differential system~\eqref{eq.dixon.diff}.
This is what motivates us to consider the general system~\eqref{eq.pqODE}
and the associated differential operator $\Delta^{\langle p,q \rangle}$.
Before us, several authors have considered the $(p,q)$-hyperbolic system~\eqref{eq.pqODE} 
or one of its specialisations as a starting point,
and have worked towards constructing generalisations of classical trigonometry.
See detailed historical references in Conrad--Flajolet's paper \cite{Conrad_06},
and also in the recent papers of Gambini, Nicoletti, and Ritelli 
\cite{Gambini_21,Gambini_24,Gambini_25}.

In the current paper, we will not attempt to generalise classical trigonometry.
Instead, we will entirely focus on combinatorial properties of the generalised hyperbolic polynomials
and the solution $(x(t),y(t))$ of the generalised hyperbolic system~\eqref{eq.pqODE}.

\subsection{Maclaurin expansions of solutions to the $(p,q)$-hyperbolic system}
\label{sec.cosh.sinh.Maclaurin}

We will use the following notation in the rest of our paper.

\begin{definition}
For $p,q\geq 2$, we use the pair $(\cosh_{p,q}(t), \sinh_{p,q}(t))$
to denote the solution of the following $(p,q)$-hyperbolic system~\eqref{eq.pqODE}
where the parameters of the system are set to  $x_0 = 1, y_0=0, a=b=1$ (and thus, $c=1, d= q/p$):
\begin{subeqnarray}
x'(t) 
&=&
y(t)^{q-1}\;,
\qquad 
x(0) = 1\;,
\\
y'(t) 
&=&
x(t)^{p-1}\;,
\qquad y(0) = 0\;.
\end{subeqnarray}

\end{definition}

Notice that when $p=q=2$,
$(\cosh_{2,2}, \sinh_{2,2}) = (\cosh,\sinh)$,
the classical hyperbolic cosine and sine,
and when $p=q=3$,
we have the Dixonian hyperbolic functions
$(\cosh_{3,3}, \sinh_{3,3}) = (\cmh,\smh)$.
Finally, notice that under the assumption $c,d>0$, the following simple rescaling
\[x(t) = (c)^{1/p}  \cosh_{p,q}(t)\,, \qquad y(t) =  (pd/q)^{1/q} \sinh_{p,q}(t) \]
gives us a local parameterisation $(x(t),y(t))$ 
for the generic curve~\eqref{eq.curve}.

We immediately notice that the generalised hyperbolic polynomials may be used to express all higher derivatives of 
$\cosh_{p,q}(t)$ and $\sinh_{p,q}(t)$,
and also obtain their Maclaurin expansions.
We record this fact in the following proposition.

\begin{proposition}
The higher derivatives of  $\cosh_{p,q}(t)$ and $\sinh_{p,q}(t)$ are given by:
\begin{subeqnarray}
 \; \dfrac{d^n}{dt^n} \cosh_{p,q}(t)
&=&
X^{\langle p,q\rangle}_n\biggl( \cosh_{p,q}(t) \,,\, \sinh_{p,q}(t)\biggr)\;,
\\
\dfrac{d^n}{dt^n} \sinh_{p,q}(t)
&=&
X^{\langle q,p\rangle}_n\biggl( \sinh_{p,q}(t) \,,\, \cosh_{p,q}(t)\biggr)\;.
\label{eq.higher.derivative}
\end{subeqnarray}
These functions have the following Maclaurin expansions:
\be
\cosh_{p,q}(t)
\;=\;
\sum_{n=0}^\infty X^{\langle p,q \rangle}_n(1,0) \dfrac{t^n}{n!}
\;,
\qquad
\sinh_{p,q}(t)
\;=\;
\sum_{n=0}^\infty X^{\langle q,p \rangle}_n(0,1) \dfrac{t^n}{n!}
\;.
\label{eq.maclaurin}
\ee
\end{proposition}
\proof
For $k,l\in \mathbb{N}$, we expand $\Delta^{\langle p,q \rangle} x^k y^l$ and 
$\dfrac{d}{dt} \cosh_{p,q}(t)^k \sinh_{p,q}(t)^l$
and comparing the two expansions
gives~(\ref{eq.higher.derivative}a).
The identity~(\ref{eq.higher.derivative}b) can then be obtained by using~\eqref{eq.xy}.
Finally, the Maclaurin expansions are obtained by setting $t=0$.
\qed

\section{Combinatorial interpretation for the generalised hyperbolic polynomials: arity-parity trees}
\label{sec.combinter}

We now construct a family of weighted trees whose weighted sum will give us the generalised hyperbolic polynomials.
Before doing this, we recall a few standard notions.

An {\em ordered tree} is a rooted tree in which the children of each vertex are linearly ordered.
An {\em increasing ordered tree} is an ordered tree in which each vertex is assigned a unique label from a totally ordered set such that the label of a vertex is always greater than the label of its parent.
We say that a vertex $v$ has {\em arity} $k\in \mathbb{N}$ if the vertex $v$ has
$k$ ordered slots for children, and for all $1\leq i\leq k$,
slot $i$ has  $0$ or $1$ child.
The {\em height} of a vertex is its distance (in terms of number of edges) 
from the root.
Thus, the root has height $0$.
For a vertex with label $i$, we use ${\rm ht}(i)$ to denote its height.

\begin{definition}
Given $k,l\in \mathbb{N}$,
a $\boldsymbol{(k,l)}$\textbfit{-arity-parity tree} of size $n$
is an increasing ordered tree with $n$ vertices (labelled $1$ through $n$) 
such that each vertex at an odd height has arity $k$,
and each vertex at an even height has arity $l$.
We use $\scrt^{\langle k,l \rangle}_n$ to denote the set of all 
$(k,l)$-arity-parity trees of size $n$.
\end{definition}

We now assign weights to arity-parity trees.
Let $T\in \scrt^{\langle k,l \rangle}_n$ be a $(k,l)$-arity-parity tree
with $n$ vertices and consider the vertex labelled  $i\in[n]$.
We introduce ${\rm wt}_T(i)$, the {\em weight of the vertex $i$ in tree $T$}
to be
\be
{\rm wt}_T(i)
\;=\;
\begin{cases}
x^{k - (\text{number of children of }i) }
\qquad \text{if ${\rm ht}(i)$ is odd}
\\
y^{l - (\text{number of children of }i) }
\qquad \text{if ${\rm ht}(i)$ is even}
\end{cases}
\ee
The weight of the tree $T$, denoted by ${\rm wt}(T)$,
 is the product of the weights of its vertices, i.e., 
${\rm wt}(T) \;=\; \prod_{i=1}^n {\rm wt}_T(i)$.
Also, we have the following alternate way of expressing the weight of a tree:
\be
{\rm wt}(T)
\;=\;
x^{k \cdot{\rm odd}(T) -  ({\rm even}(T) -1)} y^{l \cdot {\rm even}(T) - {\rm odd}(T) }
\;=\;
x^{(k+1) {\rm odd}(T) - (n-1) } y^{(l+1) {\rm even}(T) - n }
\label{eq.weight.alternate}
\ee
where ${\rm odd}(T)$ (${\rm even}(T)$) denotes the number of vertices of $T$ with odd height (resp., even height).

Here each vertex $i$ contributes a power of $x$ or $y$, 
depending on the parity of its height.
However, the exponent that $i$ contributes is always equal to its arity  minus the number of children of $i$;
this is equal to the number of {\em empty child slots} of $i$.

\begin{example}
All $(2,3)$-arity-parity trees of size $n=3$ are drawn in Figure~\ref{fig.eg}.

\begin{figure}[!ht]
\centering
\Scale[0.6]{
{
\newcommand{\nodea}{\node[draw,circle] (a) {$1$};}
\newcommand{\nodeb}{\node[draw,circle] (b) {$2$};}
\newcommand{\nodec}{\node[draw,circle] (c) {$3$};}
\newcommand{\noded}{\node (d) {$y$};}
\newcommand{\nodee}{\node (e) {$x$};}
\newcommand{\nodef}{\node (f) {$x$};}
\newcommand{\nodeg}{\node (g) {$x$};}
\newcommand{\nodeh}{\node (h) {$x$};}
\begin{tikzpicture}[scale=0.5]
\matrix[row sep = 0.5em, ampersand replacement=\&]{
       \&         \&       \&        \&     \nodea      \&    \&         \\ 
       \&     \nodeb    \&       \&         \&    \nodec     \&    \&    \noded     \\ 
   \nodee    \&         \& \nodef \& \nodeg \&\& \nodeh \& \\
};
\path[thick] (a) edge (b);
\path[thick] (a) edge (c);
\path[thick] (a) edge (d);
\path[thick] (b) edge (e);
\path[thick] (b) edge (f);
\path[thick] (c) edge (g);
\path[thick] (c) edge (h);
\end{tikzpicture}
}
\;
{
\newcommand{\nodea}{\node[draw,circle] (a) {$1$};}
\newcommand{\nodeb}{\node[draw,circle] (b) {$3$};}
\newcommand{\nodec}{\node[draw,circle] (c) {$2$};}
\newcommand{\noded}{\node (d) {$y$};}
\newcommand{\nodee}{\node (e) {$x$};}
\newcommand{\nodef}{\node (f) {$x$};}
\newcommand{\nodeg}{\node (g) {$x$};}
\newcommand{\nodeh}{\node (h) {$x$};}
\begin{tikzpicture}[scale=0.5]
\matrix[row sep = 0.5em, ampersand replacement=\&]{
       \&         \&       \&        \&     \nodea      \&    \&         \\ 
       \&     \nodeb    \&       \&         \&    \nodec     \&    \&    \noded     \\ 
   \nodee    \&         \& \nodef \& \nodeg \&\& \nodeh \& \\
};
\path[thick] (a) edge (b);
\path[thick] (a) edge (c);
\path[thick] (a) edge (d);
\path[thick] (b) edge (e);
\path[thick] (b) edge (f);
\path[thick] (c) edge (g);
\path[thick] (c) edge (h);
\end{tikzpicture}
}
{
\newcommand{\nodea}{\node[draw,circle] (a) {$1$};}
\newcommand{\nodeb}{\node[draw,circle] (b) {$2$};}
\newcommand{\nodec}{\node (c) {$y$};}
\newcommand{\noded}{\node[draw,circle] (d) {$3$};}
\newcommand{\nodee}{\node (e) {$x$};}
\newcommand{\nodef}{\node (f) {$x$};}
\newcommand{\nodeg}{\node (g) {$x$};}
\newcommand{\nodeh}{\node (h) {$x$};}
\begin{tikzpicture}[scale=0.5]
\matrix[row sep = 0.5em, ampersand replacement=\&]{
       \&         \&       \&    \nodea    \&           \&    \&         \\ 
       \&     \nodeb    \&       \&     \nodec    \&         \&  \noded   \&        \\ 
   \nodee    \&         \& \nodef \&  \&  \nodeg \& \& \nodeh  \\
};
\path[thick] (a) edge (b);
\path[thick] (a) edge (c);
\path[thick] (a) edge (d);
\path[thick] (b) edge (e);
\path[thick] (b) edge (f);
\path[thick] (d) edge (g);
\path[thick] (d) edge (h);
\end{tikzpicture}
}
\;
{
\newcommand{\nodea}{\node[draw,circle] (a) {$1$};}
\newcommand{\nodeb}{\node[draw,circle] (b) {$3$};}
\newcommand{\nodec}{\node (c) {$y$};}
\newcommand{\noded}{\node[draw,circle] (d) {$2$};}
\newcommand{\nodee}{\node (e) {$x$};}
\newcommand{\nodef}{\node (f) {$x$};}
\newcommand{\nodeg}{\node (g) {$x$};}
\newcommand{\nodeh}{\node (h) {$x$};}
\begin{tikzpicture}[scale=0.5]
\matrix[row sep = 0.5em, ampersand replacement=\&]{
       \&         \&       \&    \nodea    \&           \&    \&         \\ 
       \&     \nodeb    \&       \&     \nodec    \&         \&  \noded   \&        \\ 
   \nodee    \&         \& \nodef \&  \&  \nodeg \& \& \nodeh  \\
};
\path[thick] (a) edge (b);
\path[thick] (a) edge (c);
\path[thick] (a) edge (d);
\path[thick] (b) edge (e);
\path[thick] (b) edge (f);
\path[thick] (d) edge (g);
\path[thick] (d) edge (h);
\end{tikzpicture}
}
{
\newcommand{\nodea}{\node[draw,circle] (a) {$1$};}
\newcommand{\nodeb}{\node (b) {$y$};}
\newcommand{\nodec}{\node[draw,circle] (c) {$2$};}
\newcommand{\noded}{\node[draw,circle] (d) {$3$};}
\newcommand{\nodee}{\node (e) {$x$};}
\newcommand{\nodef}{\node (f) {$x$};}
\newcommand{\nodeg}{\node (g) {$x$};}
\newcommand{\nodeh}{\node (h) {$x$};}
\begin{tikzpicture}[scale=0.5]
\matrix[row sep = 0.5em, ampersand replacement=\&]{
       \&         \&   \nodea     \&       \&           \&    \&         \\ 
    \nodeb   \&         \&   \nodec    \&         \&         \&  \noded   \&        \\ 
       \&    \nodee     \&  \& \nodef \& \nodeg  \&  \& \nodeh  \\
};
\path[thick] (a) edge (b);
\path[thick] (a) edge (c);
\path[thick] (a) edge (d);
\path[thick] (c) edge (e);
\path[thick] (c) edge (f);
\path[thick] (d) edge (g);
\path[thick] (d) edge (h);
\end{tikzpicture}
}
{
\newcommand{\nodea}{\node[draw,circle] (a) {$1$};}
\newcommand{\nodeb}{\node (b) {$y$};}
\newcommand{\nodec}{\node[draw,circle] (c) {$3$};}
\newcommand{\noded}{\node[draw,circle] (d) {$2$};}
\newcommand{\nodee}{\node (e) {$x$};}
\newcommand{\nodef}{\node (f) {$x$};}
\newcommand{\nodeg}{\node (g) {$x$};}
\newcommand{\nodeh}{\node (h) {$x$};}
\begin{tikzpicture}[scale=0.5]
\matrix[row sep = 0.5em, ampersand replacement=\&]{
       \&         \&   \nodea     \&       \&           \&    \&         \\ 
    \nodeb   \&         \&   \nodec    \&         \&         \&  \noded   \&        \\ 
       \&    \nodee     \&  \& \nodef \& \nodeg  \&  \& \nodeh  \\
};
\path[thick] (a) edge (b);
\path[thick] (a) edge (c);
\path[thick] (a) edge (d);
\path[thick] (c) edge (e);
\path[thick] (c) edge (f);
\path[thick] (d) edge (g);
\path[thick] (d) edge (h);
\end{tikzpicture}
}
}
\vspace{3em}

\Scale[0.6]{
{
\newcommand{\nodea}{\node[draw,circle] (a) {$1$};}
\newcommand{\nodeb}{\node[draw,circle]  (b) {$2$};}
\newcommand{\nodec}{\node (c) {$y$};}
\newcommand{\noded}{\node (d) {$y$};}
\newcommand{\nodee}{\node[draw,circle]  (e) {$3$};}
\newcommand{\nodef}{\node (f) {$x$};}
\newcommand{\nodeg}{\node (g) {$y$};}
\newcommand{\nodeh}{\node (h) {$y$};}
\newcommand{\nodei}{\node (i) {$y$};}
\begin{tikzpicture}[scale=0.5]
\matrix[row sep = 0.5em, ampersand replacement=\&]{
       \&         \&        \&       \&    \nodea       \&           \\ 
       \&         \&   \nodeb    \&         \&    \nodec     \&  \noded          \\ 
       \&    \nodee     \&  \& \nodef \&  \&    \\
   \nodeg   \&  \nodeh       \&   \nodei    \&         \&         \&         \\
};
\path[thick] (a) edge (b);
\path[thick] (a) edge (c);
\path[thick] (a) edge (d);
\path[thick] (b) edge (e);
\path[thick] (b) edge (f);
\path[thick] (e) edge (g);
\path[thick] (e) edge (h);
\path[thick] (e) edge (i);
\end{tikzpicture}
}
{
\newcommand{\nodea}{\node[draw,circle] (a) {$1$};}
\newcommand{\nodeb}{\node[draw,circle]  (b) {$2$};}
\newcommand{\nodec}{\node (c) {$y$};}
\newcommand{\noded}{\node (d) {$y$};}
\newcommand{\nodee}{\node[draw,circle]  (e) {$3$};}
\newcommand{\nodef}{\node (f) {$x$};}
\newcommand{\nodeg}{\node (g) {$y$};}
\newcommand{\nodeh}{\node (h) {$y$};}
\newcommand{\nodei}{\node (i) {$y$};}
\begin{tikzpicture}[scale=0.5]
\matrix[row sep = 0.5em, ampersand replacement=\&]{
       \&         \&        \&       \&    \nodea       \&           \\ 
       \&         \&   \nodeb    \&         \&    \nodec     \&  \noded          \\ 
       \&    \nodef     \&  \& \nodee \&  \&    \\
     \&         \&   \nodeg    \&   \nodeh      \&     \nodei    \&         \\
};
\path[thick] (a) edge (b);
\path[thick] (a) edge (c);
\path[thick] (a) edge (d);
\path[thick] (b) edge (e);
\path[thick] (b) edge (f);
\path[thick] (e) edge (g);
\path[thick] (e) edge (h);
\path[thick] (e) edge (i);
\end{tikzpicture}
}
{
\newcommand{\nodea}{\node[draw,circle] (a) {$1$};}
\newcommand{\nodeb}{\node[draw,circle]  (b) {$2$};}
\newcommand{\nodec}{\node (c) {$y$};}
\newcommand{\noded}{\node (d) {$y$};}
\newcommand{\nodee}{\node[draw,circle]  (e) {$3$};}
\newcommand{\nodef}{\node (f) {$x$};}
\newcommand{\nodeg}{\node (g) {$y$};}
\newcommand{\nodeh}{\node (h) {$y$};}
\newcommand{\nodei}{\node (i) {$y$};}
\begin{tikzpicture}[scale=0.5]
\matrix[row sep = 0.5em, ampersand replacement=\&]{
       \&         \&        \&    \nodea   \&           \&           \\ 
 \nodec      \&         \&      \&      \nodeb   \&         \&  \noded          \\ 
       \&         \& \nodee \& \& \nodef   \&    \\
 \&   \nodeg   \&  \nodeh       \&   \nodei    \&         \&                  \\
};
\path[thick] (a) edge (b);
\path[thick] (a) edge (c);
\path[thick] (a) edge (d);
\path[thick] (b) edge (e);
\path[thick] (b) edge (f);
\path[thick] (e) edge (g);
\path[thick] (e) edge (h);
\path[thick] (e) edge (i);
\end{tikzpicture}
}
{
\newcommand{\nodea}{\node[draw,circle] (a) {$1$};}
\newcommand{\nodeb}{\node[draw,circle]  (b) {$2$};}
\newcommand{\nodec}{\node (c) {$y$};}
\newcommand{\noded}{\node (d) {$y$};}
\newcommand{\nodee}{\node[draw,circle]  (e) {$3$};}
\newcommand{\nodef}{\node (f) {$x$};}
\newcommand{\nodeg}{\node (g) {$y$};}
\newcommand{\nodeh}{\node (h) {$y$};}
\newcommand{\nodei}{\node (i) {$y$};}
\begin{tikzpicture}[scale=0.5]
\matrix[row sep = 0.5em, ampersand replacement=\&]{
       \&         \&        \&    \nodea   \&           \&           \\ 
 \nodec      \&         \&      \&      \nodeb   \&         \&  \noded          \\ 
       \&         \& \nodef \& \& \nodee   \&    \\
\& \& \&   \nodeg   \&  \nodeh       \&   \nodei                         \\
};
\path[thick] (a) edge (b);
\path[thick] (a) edge (c);
\path[thick] (a) edge (d);
\path[thick] (b) edge (e);
\path[thick] (b) edge (f);
\path[thick] (e) edge (g);
\path[thick] (e) edge (h);
\path[thick] (e) edge (i);
\end{tikzpicture}
}
{
\newcommand{\nodea}{\node[draw,circle] (a) {$1$};}
\newcommand{\nodeb}{\node[draw,circle]  (b) {$2$};}
\newcommand{\nodec}{\node (c) {$y$};}
\newcommand{\noded}{\node (d) {$y$};}
\newcommand{\nodee}{\node[draw,circle]  (e) {$3$};}
\newcommand{\nodef}{\node (f) {$x$};}
\newcommand{\nodeg}{\node (g) {$y$};}
\newcommand{\nodeh}{\node (h) {$y$};}
\newcommand{\nodei}{\node (i) {$y$};}
\begin{tikzpicture}[scale=0.5]
\matrix[row sep = 0.5em, ampersand replacement=\&]{
       \&    \nodea      \&        \&       \&          \&           \\ 
 \nodec      \&   \noded      \&      \&         \&      \nodeb   \&           \\ 
       \&         \&  \&  \nodee \&   \&   \nodef  \\
       \&  \&   \nodeg   \&  \nodeh       \&   \nodei          \&             \\
};
\path[thick] (a) edge (b);
\path[thick] (a) edge (c);
\path[thick] (a) edge (d);
\path[thick] (b) edge (e);
\path[thick] (b) edge (f);
\path[thick] (e) edge (g);
\path[thick] (e) edge (h);
\path[thick] (e) edge (i);
\end{tikzpicture}
}
{
\newcommand{\nodea}{\node[draw,circle] (a) {$1$};}
\newcommand{\nodeb}{\node[draw,circle]  (b) {$2$};}
\newcommand{\nodec}{\node (c) {$y$};}
\newcommand{\noded}{\node (d) {$y$};}
\newcommand{\nodee}{\node[draw,circle]  (e) {$3$};}
\newcommand{\nodef}{\node (f) {$x$};}
\newcommand{\nodeg}{\node (g) {$y$};}
\newcommand{\nodeh}{\node (h) {$y$};}
\newcommand{\nodei}{\node (i) {$y$};}
\begin{tikzpicture}[scale=0.5]
\matrix[row sep = 0.5em, ampersand replacement=\&]{
       \&    \nodea      \&        \&       \&          \&           \\ 
 \nodec      \&   \noded      \&      \&      \nodeb    \&        \&           \\ 
       \&         \&  \nodef \&   \&     \nodee  \& \\
  \&      \&  \&   \nodeg   \&  \nodeh       \&   \nodei                     \\
};
\path[thick] (a) edge (b);
\path[thick] (a) edge (c);
\path[thick] (a) edge (d);
\path[thick] (b) edge (e);
\path[thick] (b) edge (f);
\path[thick] (e) edge (g);
\path[thick] (e) edge (h);
\path[thick] (e) edge (i);
\end{tikzpicture}
}
}
\caption{All $(2,3)$-arity-parity trees of size three:
these are increasing ordered trees on the vertex set $\{1,2,3\}$
in which each vertex at an odd height has arity $2$ and each vertex at an even height has arity $3$. 
Here we have inserted the letter $x$ or $y$ associated to each empty child slot, and the weight of a vertex is the product of the weights of its empty child slots.}
\label{fig.eg}
\end{figure}
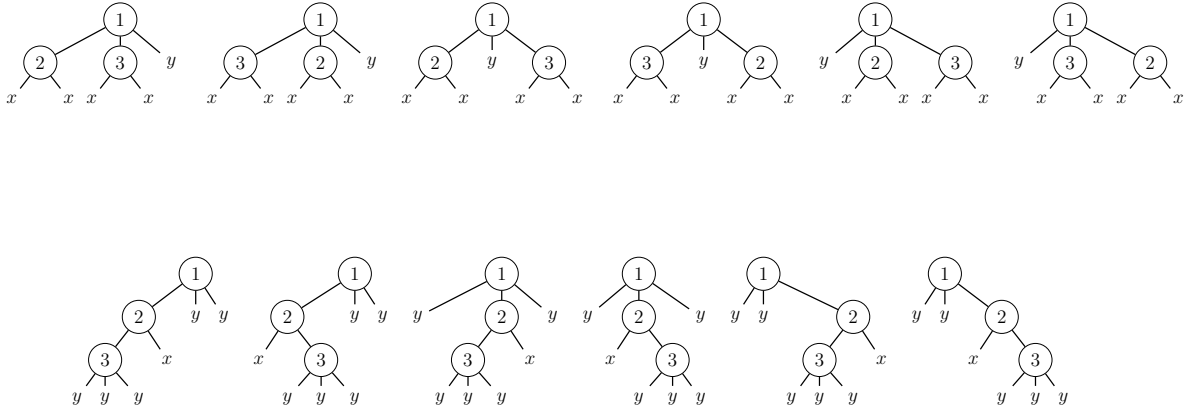

\end{example}

We are now ready to state the main theorem of this section.

\begin{theorem}
For any $p,q\in \mathbb{N}_{\geq 2}$, and $n\geq 1$,
the generalised hyperbolic polynomials $X^{\langle p,q \rangle}_n$
satisfy:
\be
X^{\langle p,q \rangle}_n(x,y)
\;=\;
\sum_{T\in \scrt^{\langle p-1,q-1 \rangle}_n}
{\rm wt}(T)
\;.
\label{eq.combinter}
\ee
\label{thm.combinter}
\end{theorem}

\proof
We will prove this theorem by applying induction and 
using Chen grammars~\cite{Chen_93}  (see also \cite{Dumont_96}).

For $n=1$, the set  $\scrt^{\langle p-1,q-1 \rangle}_1$
consists of a single tree which is just the root with arity $q-1$ and no children:
it has weight $y^{q-1}$.
Thus, the identity~\eqref{eq.combinter} is true at $n=1$.
Assuming that~\eqref{eq.combinter} holds true for $n$,
we prove it at $n+1$.

Let $T$ be a $(p-1,q-1)$-arity-parity tree on vertex set $[n]$.
Consider the various ways in which a new vertex with label $n+1$ 
can be attached to this tree as a child of some vertex $j\in [n]$
to form a new tree $T'$, (notice that $n+1$ will always be a leaf in $T'$).
The height of the vertex $j$ will be the same in both $T$ and in $T'$,
and we consider the following two cases 
depending on the parity of ${\rm ht}(j)$:
\begin{itemize}
\item If ${\rm ht}(j)$ is odd in $T$ and $T'$,
then ${\rm ht}(n+1)$ is even in $T'$.
To insert $n+1$, we select one of the empty child slots of $j$ in $T$
(which contributed a factor $x^1$ to ${\rm wt}(T)$),
and replace this slot with $n+1$.
The new vertex $n+1$ is a leaf with arity $q-1$.
Therefore, we have removed one empty child slot contributing $x$
and replaced it with $q-1$ new empty child slots each contributing $y$.
Thus, we have the grammar rule $x \mapsto y^{q-1}$.

\item If ${\rm ht}(j)$ is even in $T$ and $T'$,
then ${\rm ht}(n+1)$ is odd in $T'$.
Here, we instead have removed one empty child slot contributing $y$
and replaced it with $p-1$ new empty child slots each contributing $x$.
Thus, we have the grammar rule $y \mapsto x^{p-1}$.

\end{itemize}
We have thus established
\be
\sum_{T\in \scrt^{\langle p-1,q-1 \rangle}_{n+1}}
{\rm wt}(T)
\;=\;
\left(x^{p-1}\dfrac{d}{dy} + y^{q-1}\dfrac{d}{dx}\right)
\sum_{T\in \scrt^{\langle p-1,q-1 \rangle}_{n}}
{\rm wt}(T)
\;=\;
\Delta X^{\langle p,q \rangle}_n
\;=\;
X^{\langle p,q\rangle}_{n+1}
\;,
\ee
which finishes the proof. \qed

\begin{example}
In Figure~\ref{fig.eg}, the total sum of weights is $6x^4y+ 6 x y^5$,
which is also equal to $X^{\langle 3,4 \rangle}_3(x,y)$.
\end{example}

{\bf Remark.}
We later realised that one can come up with this construction by following the general theory
of Leroux and Viennot \cite{Leroux_86} on combinatorial theory of ordinary differential equations.
Also see {\em historiographers theory} in \cite{Randrianirina_25} and references therein.

Tree statistics according to the parity of vertex heights have already been considered by other authors such as in \cite{Deutsch_00,Lin_25}.
\myendremark

Next, we will use our combinatorial interpretation for the generalised hyperbolic polynomials to obtain 
combinatorial interpretations of the Maclaurin coefficients of the functions $\cosh_{p,q}$ and $\sinh_{p,q}$.
To do this, let $\scrc^{\langle k,l \rangle}_n \subseteq \scrt^{\langle k,l \rangle}_n$ 
be the subclass of $(k,l)$-arity-parity trees in which each vertex at an even height has exactly $l$ children.
Also, let $\scrs^{\langle k,l \rangle}_n \subseteq \scrt^{\langle k,l \rangle}_n$ 
be the subclass of $(k,l)$-arity-parity trees in which each vertex at an odd height has exactly $k$ children.

\begin{theorem}
\begin{itemize}
\item[(a)] The function $\cosh_{p,q}(t)$ 
is the exponential generating function of $(p-1,q-1)$-arity parity trees in which each vertex at an even height has exactly 
$q-1$ children:
\be
\cosh_{p,q}(t) 
\;=\;
1+ \sum_{n=1}^\infty \left|\scrc^{\langle p-1,q-1 \rangle}_{q n}\right| \dfrac{t^{q n}}{(q n)!}
\;.
\ee

\item[(b)] The function $\sinh_{p,q}(t)$ 
is the exponential generating function of $(q-1,p-1)$-arity-parity trees in which each vertex at an odd height
has exactly $q-1$ children:
\be
\sinh_{p,q}(t)
\;=\;
\sum_{n=0}^\infty
\left|\scrs^{\langle q-1,p-1 \rangle}_{q n+1}\right|
 \dfrac{t^{qn+1}}{(qn+1)!}
\;.
\ee
\end{itemize}
\label{thm.cosh.sinh.combinter}
\end{theorem}

\proof
For each $T\in \scrc^{\langle k,l \rangle}_n$,
notice that ${\rm wt}(T)$ must be a monomial in $x$,
and conversely, each $T\in \scrt^{\langle k,l \rangle}_n$ whose weight is a monomial in $x$ 
must belong to $\scrc^{\langle k,l \rangle}_n$.
This along with Theorem~\ref{thm.combinter} gives us for $n\geq 1$
\be
X^{\langle k+1,l+1 \rangle}_n(1,0) 
\;= \;
\left|\scrc^{\langle k,l \rangle}_{n}\right|
\;.
\ee
A similar argument where we avoid any factor of $x$ in the weight of trees gives us for $n\geq 0$
\be
X^{\langle k+1,l+1 \rangle}_n(0,1) 
\;= \;
\left|\scrs^{\langle k,l \rangle}_{n}\right|
\;.
\ee

Finally, notice that any tree in which each vertex at even height has  exactly $l$ children 
must satisfy $l+1 | n$.
Also, any tree in which each vertex at odd height has  exactly $k$ children 
must satisfy $n \equiv 1 \pmod{k+1}$.
Putting all of these arguments together along with equation~\eqref{eq.maclaurin} finishes the proof.
\qed

\subsection{Some special cases}

We now elaborate on some special cases and the consequences of the combinatorial interpretation just constructed.

\noindent
\underline{When $p=q=2$:}

Here, the set $\scrt^{\langle 1,1 \rangle}_n$ consists of the unique unary tree on $n$ vertices
that has weight $x$ when $n$ is even and $y$ when $n$ is odd.
This is in agreement with the fact 
\[
X^{\langle 2,2 \rangle}_{2n}(x,y) = x\;, \qquad X^{\langle 2,2 \rangle}_{2n+1}(x,y) = y\;.
\]
Theorem~\ref{thm.cosh.sinh.combinter} gives us the known Maclaurin expansions
\be
\cosh(t) \;=\; \sum_{n=0}^\infty \dfrac{t^{2n}}{(2n)!}\;,
\qquad
\sinh(t) \;=\; \sum_{n=0}^\infty \dfrac{t^{2n+1}}{(2n+1)!}\;.
\ee

\medskip

\noindent
\underline{When $p=q=3$:}

We first list out the polynomials $X_n^{\langle 3,3 \rangle}$ for the first few values of $n$:
\begin{subeqnarray}
	X_0^{\langle 3,3 \rangle}(x,y) &=& x \\
	X_1^{\langle 3,3 \rangle}(x,y) &=& y^2 \\
	X_2^{\langle 3,3 \rangle}(x,y) &=& 2x^2y \\
	X_3^{\langle 3,3 \rangle}(x,y) &=& 2 x (x^3+2 y ^3)\\
	X_4^{\langle 3,3 \rangle}(x,y) &=& 4 y^2 (5 x^3 + y^3)\\
	X_5^{\langle 3,3 \rangle}(x,y) &=& 40 x^2 y (x^3+2 y ^3)
\end{subeqnarray}
We will refer to these as the {\em Dixon-Schett polynomials}.

In the combinatorial interpretation of Theorem~\ref{thm.combinter},
each vertex here has arity $2$ and therefore, $(2,2)$-arity-parity trees are simply increasing binary trees,
and it is well known that these are in bijection with permutations.
Therefore, we have the cardinality
\be 
X_n^{\langle 3,3\rangle}(1,1)
\;=\;
\left|\scrt^{\langle 3,3 \rangle}_n\right| 
\;=\;
n!
\ee
which we may verify for $n\leq 5$ from the data above.

Let $T$ be an increasing binary tree of size $n$.
Let ${\rm odd}_i(T)$ (${\rm even}_i(T)$) denote the number of vertices at an odd height (even height, resp.)
which have exactly $i$ children.
Thus, the weight of a binary tree $T$ can be expressed as follows:
\be
{\rm wt}(T)
\;=\;
x^{{\rm odd}_1(T) + 2{\rm odd}_0(T)} y^{{\rm even}_1(T) + 2{\rm even}_0(T)}\;.
\ee

Here $\cosh_{3,3}(t) = {\rm cmh}(t)$, $\sinh_{3,3}(t) = {\rm smh}(t)$
are the hyperbolic analogs of the elliptic functions sm, cm first studied by Dixon \cite{Dixon_1890},
and they are now known as the Dixonian elliptic functions.
These were further studied by Conrad and Flajolet~\cite{Conrad_06} who developed their combinatorial theory.
In fact, Theorem~\ref{thm.cosh.sinh.combinter} at $p=q=3$ exactly gives us the tree interpretation 
of the coefficients of ${\rm smh}, {\rm cmh}$ constructed by Conrad and Flajolet in \cite[Theorem~4]{Conrad_06}.

The coefficients of ${\rm cosh_{3,3}}$ are \cite[A104134]{OEIS} (upto sign)
and those of $\sinh_{3,3}$ are \cite[A104133]{OEIS} (upto sign).

\medskip

\noindent
\underline{When $p=q$:}

In the general $p=q$ case, we see that $\scrt^{\langle p-1,p-1 \rangle}_n$ consists of all increasing $(p-1)$-ary trees.
It is well known \cite[p. 30, Example 1]{Bergeron_92} that the number of increasing $(p-1)$-ary trees on $n$ vertices is the multifactorial
$\prod_{j=0}^{n-1} (1+ (p-2) j)$.
We record this fact as a proposition.

\begin{proposition}
The $(p,p)$-hyperbolic polynomials under the specialisation $x=y=1$ count the number of increasing $(p-1)$-ary trees as follows:
\be
X^{\langle p,p \rangle}_n(1,1)
\;=\;
|\scrt^{\langle p-1,p-1 \rangle}_n|
\;=\;
\prod_{j=0}^{n-1} (1+ (p-2) j)\;.
\ee
\end{proposition}

Next, we return to the interpretations of $\cosh_{p,p},\sinh_{p,p}$ in Theorem~\ref{thm.cosh.sinh.combinter}.
The fact that for the $n$th Maclaurin coefficient to be non-zero, $n$ must satisfy $p|n$ and $p|n-1$, respectively,
was previously noticed by Gambini, Nicoletti, and Ritelli in \cite[Theorem~2.1]{Gambini_25}.
Here, we additionally have a combinatorial interpretation for the numbers
$|{}_p c_i |, |{}_p s_i |$ in their theorem.

For $p=q=4$, we now mention a curious statement by B\l asiak and Flajolet in \cite{Blasiak_11}.
In the last row of the table in \cite[Figure~24]{Blasiak_11},
the authors say that the operator $y^3 \partial_x + x^3 \partial_y$
``corresponds to parity of levels in increasing ternary trees'' without any further description,
but then pointed to the references \cite{Flajolet_05, Flajolet_06}.
However, we did not find any mention of ternary trees in these latter papers.
We suspect that the authors of \cite{Blasiak_11} must have known the exact same combinatorial interpretation as ours.

The coefficients of ${\rm cosh_{4,4}}$ are in \cite[A153300]{OEIS}
and those of $\sinh_{4,4}$ are in \cite[A153301]{OEIS}.

\medskip

\noindent
\underline{When $p=2$ or $q=2$:}

\begin{proposition}
For all $n\geq 0$, the $(p,2)$-hyperbolic polynomials and the $(2,p)$-hyperbolic polynomials satisfy the following relation:
\be
X^{\langle p,2 \rangle}_{n+1}(x,y)
\;=\;
X^{\langle 2,p \rangle}_{n}(y,x)
\;.
\label{eq.p2q2id}
\ee
\end{proposition}
\proof
For $n=0$, both sides of equation~\eqref{eq.p2q2id} are equal to $y$.
We use Theorem~\ref{thm.combinter} to provide a bijective proof for $n\geq 1$.
Consider a tree $T\in \scrt^{\langle p-1 ,1\rangle}_{n+1}$;
its root has arity $1$.
Removing the root and then subtracting $1$ from all the remaining vertex labels gives us a tree 
$T'\in \scrt^{\langle 1,p-1 \rangle}_n$; this is clearly a bijection between the two sets 
$\scrt^{\langle p-1 ,1\rangle}_{n+1}$ and $\scrt^{\langle 1,p-1\rangle}_{n}$
and the weights are preserved after interchanging $x$ and $y$.
This finishes the proof.
\qed

As a consequence, we get that the sequences 
$(X^{\langle p,2 \rangle}_{n}(1,1))_{n\geq 0}$ and $(X^{\langle 2,p \rangle}_{n}(1,1))_{n\geq 0}$
are just shifts of each other.
From equation~\eqref{eq.maclaurin},
we also get that the coefficients of $\cosh_{p,2}(t)$, $\sinh_{p,2}(t)$ are shifts of each other,
and so are the coefficients of $\cosh_{2,q}(t)$, $\sinh_{2,q}(t)$;
these facts are evidenced by the identities $(\cosh_{p,2})' = \sinh_{p,2}$
and $(\sinh_{2,q})' = \cosh_{2,q}$.

\medskip

For the special case $p=3,q=2$, recall that the corresponding $(2,1)$-arity-parity trees
are increasing trees in which each vertex at an odd height has arity $2$, and  at an even height has arity $1$.
These are precisely the ``reduced 3-historic trees'' of Burghart and Wagner \cite{Burghart_24}.
This sequence is \cite[A007558]{OEIS}.

\medskip

After this, we show that the coefficients of $\cosh_{p,2}(t)$ also enumerate increasing bilabelled trees of Kuba and Panholzer \cite{Kuba_12, Kuba_16}.
Recall that an {\bf \em increasing bilabelled $\bm{k}$-ary tree} on $n$ vertices
is a $k$-ary ordered tree having $n$ vertices in which each vertex is assigned two different labels
and furthermore, the label sets of different vertices are disjoint, such that both labels of a vertex
are bigger than both labels of its parent. 
We use $\scri^{(k)}_{n}$ to denote the set of increasing bilabelled $k$-ary trees on $n$ vertices 
whose labels are assigned from the set $[2n] = \{1,\ldots, 2n\}$.

\begin{proposition}
The set $\scri^{(k)}_n$ of increasing bilabelled $k$-ary trees on $n$ vertices
is in bijection with the set  $\scrc^{\langle k,1\rangle}_{2n}$ of 
$(k,1)$-arity-parity trees such that each vertex at even height has exactly $1$ child.

Furthermore, let 
\begin{eqnarray}
F(t)
&=&
\sum_{n=1}^\infty 
\left| \scri^{(k)}_{n}\right| \dfrac{t^{2n}}{(2n)!} 
\;=\;
\sum_{n=1}^\infty 
\left| \scrc^{\langle k,1 \rangle}_{2n}\right| \dfrac{t^{2n}}{(2n)!} 
\\
&=&
\sum_{n=1}^\infty 
X^{\langle k+1,2 \rangle}_{2n}(1,0) \dfrac{t^{2n}}{(2n)!} 
\;=\;
\cosh_{k+1,2}(t)-1
\;.
\end{eqnarray}
This exponential generating function satisfies the first order differential equation
\be
(F'(t))^2
\;=\;
2 \Phi(F(t)) 
\;,
\qquad
F(0) = 0\;,
\ee
where $\Phi(s) = \sum_{j=0}^k \frac{1}{j+1}\binom{k}{j} s^{j+1}$.
\end{proposition}

\proof
Start with a tree $T\in \scrc^{\langle k,1\rangle}_{2n}$ and for every vertex $v$ at an even height,
contract the edge between $v$ and its child and assign the two labels of $v$ and its child 
to the new merged vertex.
The resulting tree $T\in \scri^{(k)}_n$, and as this process is reversible, this is a bijection.

The fact that the differential equation here is satisfied by $F(t)$, follows from \cite[Proposition~1]{Kuba_16}.
\qed

The coefficients of $\sinh_{3,2}(t)$ and $\cosh_{3,2}(t)$ are in \cite[A063902]{OEIS}.

\subsection{Exponential generating function}

Consider the exponential generating function 
for the $(p,q)$-hyperbolic polynomials:
\be 
\scrx^{\langle p,q \rangle}(t;x,y)
\;\eqdef\;
\sum_{n=0}^\infty
X^{\langle p,q \rangle}_n(x,y) \dfrac{t^n}{n!}
\;.
\ee
We now use our combinatorial interpretation in Theorem~\ref{thm.combinter}
and observe the following.

\begin{proposition}
Consider the exponential generating function $\scrx^{\langle p,q \rangle}(t;x,y)$
and set $\scry^{\langle p,q \rangle}(t;x,y) \coloneqq \scrx^{\langle q,p \rangle}(t;y,x)$.
The following statements are true:
\begin{itemize}
\item[(a)]
The pair $(\scrx^{\langle p,q \rangle}, \scry^{\langle p,q \rangle})$
is a solution $(F(t),G(t))$ to the following $(p,q)$-hyperbolic system:
\begin{subeqnarray}
F'
&=&
G^{q-1}
\;,
\qquad F(0) = x\;,
\\
G'
&=&
F^{p-1}\;,
\qquad G(0)=y\;.
\label{eq.egf.ODE}
\end{subeqnarray}

\item[(b)] 
The pair $(\scrx^{\langle p,q \rangle}, \scry^{\langle p,q \rangle})$ satisfies the following equation:
\be
q\scrx^{\langle p,q \rangle}(t;x,y)^p 
\,-\,  
p\scry^{\langle p,q \rangle}(t;x,y)^q
\;=\;
qx^p-py^q
\;.
\label{eq.egf.curve}
\ee

\end{itemize}
\end{proposition}

\proof
(a) From first principles of enumerative combinatorics, it clearly follows that the system~\eqref{eq.egf.ODE} 
is satisfied.
Indeed, if we start with a $(p-1,q-1)$-arity-parity tree $T$
and remove the root, we obtain an ordered $(q-1)$-tuple of trees;
each of these are either empty with weight $y$, or are a $(q-1,p-1)$-arity-parity tree of size at least $1$ but with the weights $x$ and $y$ swapped.
Similarly, if we start with a $(q-1,p-1)$-arity-parity tree $T'$
and remove the root, we obtain an ordered $(p-1)$-tuple of trees;
each of these are either empty with weight $x$, or are a $(p-1,q-1)$-arity-parity tree of size at least $1$.

(b) As the pair $(\scrx^{\langle p,q \rangle}, \scry^{\langle p,q \rangle})$ satisfies~\eqref{eq.egf.ODE},
we have
\be
(q \scrx^{\langle p,q \rangle}(t;x,y)^p )'
\;=\;
(p \scry^{\langle p,q \rangle}(t;x,y)^q )'
\;=\;
pq \scrx^{\langle p,q \rangle}(t;x,y)^{p-1} 
\scry^{\langle p,q \rangle}(t;x,y)^{q-1}
\ee
Thus, the difference in the left-hand-side of~\eqref{eq.egf.curve} is a constant,
and this constant  can be obtained by plugging $t=0$.
%

\qed

\section{Reduced hyperbolic polynomials}
\label{sec.reduced}

In this section, we fix some $p,q \in \mathbb{N}_{\geq 2}$,
and drop the superscript $\langle p,q \rangle$,
and denote our operator simply by $\Delta$,
and the $(p,q)$-hyperbolic polynomials as $X_n(x,y)$.
Also, we will use the symbols $\sfX, \sfY$ 
to denote 
$\sfX \eqdef x^p$,
$\sfY \eqdef y^q$.

We have the following theorem:

\begin{theorem}
Fix $n,p,q\in \mathbb{N}$ with $n\geq 0$ and $p,q\geq 2$,
and let $r_n,s_n\in \N$ be the smallest nonnegative integers such that
\hbox{$r_n\equiv 1-n \pmod{p}$}
and
\hbox{$s_n\equiv -n \pmod{q}$.}
The following statements hold:
\begin{itemize}
\item[(i)]
The quantity $\dfrac{1}{x^{r_n} y^{s_n}} X_n(x,y)$
is a homogeneous polynomial in the variables 
$\sfX = x^p, \sfY  = y^q$.

\item[(ii)]
Let $\Xhat_n(\sfX, \sfY) \eqdef \dfrac{1}{x^{r_n} y^{s_n}} X_n(x,y)$
be the homogeneous polynomial in part (i)
in the variables $\sfX,\sfY$.
These polynomials satisfy the following recurrences:
\begin{eqnarray}
\hspace*{-10mm}
\Xhat_{n+1}(\sfX,\sfY)
\;=\;
\begin{cases}
(r_n \sfY + s_n \sfX) \Xhat_n(\sfX,\sfY) 
\nonumber\\
\,+\,
\sfX\sfY 
\left(
p\dfrac{\partial}{\partial \sfX}
\,+\,
q\dfrac{\partial}{\partial \sfY}
\right)
\Xhat_n(X,Y)
\qquad \text{when $r_n,s_n>0$}
\\[4mm]
s_n \Xhat_n(\sfX,\sfY)
\,+\,
\sfY
\left(
p\dfrac{\partial}{\partial \sfX}
\,+\,
q\dfrac{\partial}{\partial \sfY}
\right)
\Xhat_n(\sfX,\sfY)
\qquad \text{when $r_n=0,s_n>0$}
\\[4mm]
r_n  \Xhat_n(\sfX,\sfY)
\,+\,
\sfX
\left(
p\dfrac{\partial}{\partial \sfX}
\,+\,
q\dfrac{\partial}{\partial \sfY}
\right)
\Xhat_n(\sfX,\sfY)
\qquad \text{when $r_n>0,s_n=0$}
\\[4mm]
\left(
p\dfrac{\partial}{\partial \sfX}
\,+\,
q\dfrac{\partial}{\partial \sfY}
\right)
\Xhat_n(\sfX,\sfY)
\qquad \text{when $r_n=s_n=0$}
\end{cases}
\nonumber\\
\label{eq.Xhat.recurrences}
\end{eqnarray}
with initial condition $\Xhat_0(\sfX,\sfY) = 1$\;.

\end{itemize}
\label{prop.realrooted}
\end{theorem}

For a fixed $p,q$ and varying $n$, we will refer to the homogeneous bivariate polynomials $\Xhat_n(\sfX, \sfY)$
( $ = \Xhat^{\langle p,q \rangle}_n(\sfX, \sfY)$)
in Theorem~\ref{prop.realrooted}
as the {\bf \em $\bm{(p,q)}$-reduced hyperbolic polynomials}.

Before proceeding to prove Theorem~\ref{prop.realrooted}, 
we list out the first few values of $\Xhat_n(x,y)$
for $0\leq n\leq 3$.
First notice that for any $p,q\in \mathbb{N}_{\geq 3}$, we must have the following values of $r_n,s_n$ for $0\leq n\leq 3$:
\be
(r_0, s_0)= (1,0)\;, (r_1,s_1) = (0,q-1) \;, (r_2,s_2) = (p-1,q-2) \;, (r_3,s_3) = (p-2,  q-3) \;.
\ee
Using this, we can compute our first few $(p,q)$-reduced hyperbolic polynomials for $p,q\geq 3$:
\begin{subeqnarray}
	\Xhat_0(x,y) &=& 1\\
	\Xhat_1(x,y) &=& 1\\
	\Xhat_2(x,y) &=& (q-1)\\
	\Xhat_3(x,y) &=& (q-1)\left[
		(q-2) x \,+\, (p-1) y
		\right]
\label{eq.compute.xhat0123}
\end{subeqnarray}

\medskip

In order to prove Theorem~\ref{prop.realrooted},
we require the following easy observation.

\begin{lemma}
Let $f(x,y)\in \Z[x,y]$ be a homogeneous polynomial. 
Then we have
\be
	x y \Delta f(x^p,y^q)
	\;=\;
	\left. 
	x y 
	\left(
	p \dfrac{\partial}{\partial x} + q \dfrac{\partial}{\partial y}
	\right)
	f(x,y) 
	\right|_{x\rightarrow x^p, y\rightarrow y^q}\;.
	\label{eq.homogenous.delta}
\ee
\label{lem.homogenous.delta}
\end{lemma}

\proof
It suffices to verify equation~\reff{eq.homogenous.delta}
for a monomial $x^a y^b$.
Indeed, this is not hard to verify as
both sides of~\reff{eq.homogenous.delta}
are equal to 
$ap\, x^{ap} y^{(b+1)q} \,+\, bq\, x^{(a+1)p}y^{bq}$.
\qed

We are now ready to prove Theorem~\ref{prop.realrooted}.

\proofof{Theorem~\ref{prop.realrooted}}
We will use induction to simultaneously prove (i) and (ii).

For the base case at $n=0$,
we certainly get $\Xhat_{0}(\sfX,\sfY) = 1$
and here $\Xhat_0(\sfX,\sfY)$ is a homogeneous polynomial in $\sfX,\sfY$.

Let us assume now that $\Xhat_n$
is a homogeneous polynomial in $\sfX,\sfY$ for some $n$.
We write 
\begin{subeqnarray}
	\Xhat_{n+1}(\sfX,\sfY) 
	&=& \dfrac{1}{x^{r_{n+1}} y^{s_{n+1}}}\Delta X_{n}(x,y)\\
	&=& \dfrac{\Delta \bigl(x^{r_n} y^{s_n}\bigr)}{
		x^{r_{n+1}} y^{s_{n+1}}} \,
	\Xhat_{n}(x^p,y^q)
	\,+\,
	\dfrac{x^{r_n} y^{s_n}}{x^{r_{n+1}} y^{s_{n+1}}} \, 
	\Delta \Xhat_{n}(x^p,y^q)
\label{eq.xnplus1.rec}
\end{subeqnarray}

We now have a few different cases depending on whether $r_n$ and $s_n$ 
are equal to zero or are non-zero:\\
\begin{itemize}
\item[(i)] If $r_n,s_n>0$,
	we have from the hypothesis
	$r_{n+1} = r_n -1,s_{n+1} = s_n -1$.
	We substitute this in 
	(\ref{eq.xnplus1.rec}b)
	and then apply Lemma~\ref{lem.homogenous.delta}
	to obtain
	\be
	 \Xhat_{n+1}(\sfX, \sfY ) 
	\;:=\;
        (r_n \sfY + s_n \sfX) \Xhat_{n}(\sfX,\sfY)
        \,+\,
        \sfX\sfY
	\left(
        p \dfrac{\partial}{\partial \sfX} + q \dfrac{\partial}{\partial \sfY}
        \right)
	\Xhat_{n}(\sfX,\sfY)
	\;.
	\label{eq.xhat.rs.nonzero}
	\ee
	Since $\Xhat_{n}(\sfX,\sfY)$ is homogeneous,
	so is 
	$\Xhat_{n+1}(\sfX,\sfY)$ from~\eqref{eq.xhat.rs.nonzero}.

\item[(ii)] 
	If $r_n=0$ and $s_n>0$,
	we have from the hypothesis 
	$r_{n+1} = p-1$, $s_{n+1} = s_n-1$.
	We substitute this in
        (\ref{eq.xnplus1.rec}b)
        and then apply Lemma~\ref{lem.homogenous.delta}
        to obtain
	\be
	 \Xhat_{n+1}(\sfX, \sfY )
        \;:=\;
        s_n \Xhat_{n}(\sfX,\sfY)
        \,+\,
        \sfY
        \left(
        p \dfrac{\partial}{\partial \sfX} + q \dfrac{\partial}{\partial \sfY}
        \right)
        \Xhat_{n}(\sfX,\sfY)
        \;.
	\label{eq.xhat.r=0.s.nonzero}
	\ee
	Here, we deduce that $\Xhat_{n+1}(x,y)$ is homogeneous.
	
\end{itemize}
We leave out the details of the other two cases
(when $r_n>0,s_n=0$ and when $r_n=s_n=0$)
as they are similar to the previous two cases.
In all cases, we see that the recurrences~\eqref{eq.Xhat.recurrences}
hold and we deduce that the polynomials $\Xhat_{n+1}$ are also 
homogeneous.
This finishes the proof.
\qed

We have the following corollary.
\begin{corollary}
The $(p,q)$-reduced hyperbolic polynomial satisfies the following identity 
in terms of $(p-1,q-1)$-arity-trees:
\be
\Xhat_n^{\langle p,q \rangle}(\sfX,\sfY)
\;=\;
\sum_{T\in \scrt^{\langle p-1,q-1 \rangle}_n} \sfX^{{\rm odd}(T) - \lceil (n-1)/p \rceil} \sfY^{{\rm even}(T) - \lceil n/q \rceil}
\;.
\label{eq.reduced.combinter}
\ee
Thus, the degree of the reduced hyperbolic polynomial is given by
\be
\deg \Xhat_n^{\langle p,q \rangle}
\;=\;
n  \,-\, \lceil (n-1)/p \rceil \,-\, \lceil n/q \rceil \;.
\ee
\label{cor.reduced.combinter}
\end{corollary}

\proof
Using Theorem~\ref{thm.combinter} and equation~\eqref{eq.weight.alternate},
we have
\be
X^{\langle p,q \rangle}_n
\;=\;
\sum_{T\in \scrt^{\langle p-1,q-1 \rangle}_n} 
x^{p \cdot{\rm odd}(T) -  (n-1) } y^{q\cdot {\rm even}(T) - n}\;.
\ee
Dividing out $x^{r_n} y^{s_n}$ and then writing each factor in terms of $\sfX$, $\sfY$ gives us~\eqref{eq.reduced.combinter}, and the expression for the degree is immediate from this.
\qed

We end this section with the following remark.

\noindent
{\bf Remark.}
Using the recurrence~\eqref{eq.Xhat.recurrences},
we can write out recurrences between
the total degrees of the polynomials
$\Xhat_n(\sfX,\sfY)$ in the variables $\sfX,\sfY$:
\begin{eqnarray}
      {\rm deg} \Xhat_{n+1}(\sfX,\sfY) 
      \;=\;
      \begin{cases}
      1 + {\rm deg} \Xhat_n(\sfX,\sfY)
      \qquad
      \text{if }r_n,s_n>0
      \\[4mm]
      {\rm deg} \Xhat_n(\sfX,\sfY)
      \qquad\qquad
      \text{if }r_n=0 \text{ or }s_n=0 \text{ but not both}
      \\[4mm]
      -1 + {\rm deg} \Xhat_n(\sfX,\sfY)
      \qquad
      \text{if }r_n=s_n=0
      \end{cases}
      \nonumber
      \\
\label{eq.degree.recurrences}
\end{eqnarray}
When $p=q$, the recurrence~\eqref{eq.degree.recurrences} becomes
\begin{eqnarray}
{\rm deg} \Xhat_{n+1}(\sfX,\sfY) 
      \;=\;
      \begin{cases}
      1 + {\rm deg} \Xhat_n(\sfX,\sfY)
      \qquad
      \text{if } p\nmid n,n-1
      \\[4mm]
      {\rm deg} \Xhat_n(\sfX,\sfY)
      \qquad\qquad
      \text{if } p\mid n \text{ or } p\mid n-1
      \end{cases}
      \nonumber
      \\
\end{eqnarray}
From this, we then obtain for all $k\geq 0$
and $0\leq l<p$ we have
\be
{\rm deg} \Xhat_{pk+l}
\;=\;
\begin{cases}
(p-2)k \qquad\qquad\qquad\: \text{ if }l=0,1,2
\\
(p-2) k + (l-2) \qquad \text{ if }2<l< p 
\end{cases}
\ee
\myendremark

\section{Real-rootedness of the univariate reduced hyperbolic polynomials}
\label{sec.realrooted}

In this section, we assume $p,q\geq 3$ for simplicity and drop the superscript $\langle p,q \rangle$,
and denote our operator simply by $\Delta$,
and the $(p,q)$-reduced hyperbolic polynomials as $\Xhat_n(x,y)$.
Our second main contribution of this paper is Theorem~\ref{thm.xhat.realrooted} which we prove now.
%

Our proof of Theorem~\ref{thm.xhat.realrooted} will rely on 
two of Liu and Wang's criteria for real-rootedness \cite[Theorem~2.1~and~Lemma~2.5(ii)]{Liu_06}.
We recall these for the benefit of the reader.

\begin{theorem}{\cite{Liu_06}}
(a)
Let $F, f , g$ be three real polynomials satisfying the following conditions.
\begin{itemize} 
\item $F(x) = a(x)f (x)+b(x)g(x)$, where $a(x),b(x)$ are two real polynomials, such that 
${\rm deg}F = {\rm deg}f$ or ${\rm deg}f +1$.
\item $f,g$ are real-rooted and $g$ weakly interlaces $f$.

\item $F$ and $g$ have leading coefficients of the same sign.

\end{itemize}
Suppose that $b(x)\leq 0$ whenever $f(x) = 0$. Then $F$ is real-rooted and $f$ weakly interlaces  
$F$. 

(b)
Let $G(x) = a(x)f (x) + b(x)g(x)$ where $G, f, g$ are standard (either identically zero or their leading coefficient is positive)
and $a(x), b(x)$ are real polynomials. 
Suppose that $f, g $ real-rooted and $g$ strictly interlaces $f$. 
If ${\rm deg}G\leq  {\rm deg}f$ and $b(x) > 0$ whenever $f (x) = 0$, then $G$ is real-rooted and $G$ strictly interlaces $f$.
\label{thm.liuwang}
\end{theorem}
See \cite{Liu_06} for all details.

Before proceeding to the proof, we also require the following easy lemma.
\begin{lemma}
Let $f(x,y)\in \C[x,y]$ be a homogeneous polynomial of total degree $d$
in the variables $x,y$.
We then have
\be
        \left.\dfrac{\partial}{\partial y} f(x,y)\right|_{y\rightarrow 1}
        \;=\;
        \left(d - x \dfrac{\partial}{\partial x}\right) f(x,1).
\label{eq.homogeneous.dy}
\ee
\label{lem.homogeneous.dy}
\end{lemma}

\proof
It suffices to verify~\reff{eq.homogeneous.dy}
for a monomial $x^a y^{d-a}$.
Indeed, we can see that both sides of~\reff{eq.homogeneous.dy}
are equal to~$(d-a)x^{a}$.
\qed

\proofof{Theorem~\ref{thm.xhat.realrooted}}
In this proof we will use $\Xhat_{n}(x)$
to denote $\Xhat_{n}(x,1)$ and use induction.

For the base case, we look at~\eqref{eq.compute.xhat0123},
and immediately see that the zero of the linear polynomial $\Xhat_3$ is real and non-positive. 

By induction, we assume that $\Xhat_{n}(x)$
is real-rooted and all its zeroes are simple and negative.
Also, let $r_n,s_n$ be as Theorem~\ref{prop.realrooted},
and let $d$ be the total degree of $\Xhat_{n}(x,y)$ in both variables $x,y$.
As in the proof of Theorem~\ref{prop.realrooted},
we have four different cases depending on whether 
$r_n$ and $s_n$ are zero or not:
\begin{itemize}
	\item When $r_n,s_n>0$,
	using~\reff{eq.Xhat.recurrences}
	and Lemma~\ref{lem.homogeneous.dy}, we obtain
	\begin{eqnarray}
        \Xhat_{n+1}(x) 
	&=&
	\left[
	(r_n +s_n x)
	\,+\, p x\dfrac{\partial}{\partial x}
        \,+\,
	q x
	\left(
	d \,-\,
	x \dfrac{\partial}{\partial x}  
        \right)
	\right]
        \Xhat_{n}(x)
        \nonumber\\
	&=&
	[r_n + (s_n+qd)x] \Xhat_n(x)
	\,+\,
	x (p - q x) \Xhat_n'(x)\;.
	\end{eqnarray}
We set $F(x) = \Xhat_{n+1}(x)$, $f(x) = \Xhat_n(x)$
and $g(x) = \Xhat_n'(x)$
in 
Theorem~\ref{thm.liuwang}(a)
and notice that $b(x) := x(p - q x)$
satisfies $b(x)\leq 0$ for all $x\leq 0$.

\item When $r_n>0$ and $s_n=0$,
using~\reff{eq.Xhat.recurrences}
and Lemma~\ref{lem.homogeneous.dy}, we obtain
\begin{eqnarray}
        \Xhat_{n+1}(x) 
        &=&
	\left[r_n \,+\, p x\dfrac{\partial}{\partial x}
	\,+\, 
	q x \left(d - x\dfrac{\partial}{\partial x}\right)
	\right]
	\Xhat_{n}(x) 
	\nonumber\\
	&=&
	(r_n + qdx)\Xhat_n(x) 
	\,+\,
	x(p - q x) \Xhat_n'(x) 
\end{eqnarray}
Here again, we set $F(x) = \Xhat_{n+1}(x)$, $f(x) = \Xhat_n(x)$
and $g(x) = \Xhat_n'(x)$
in 
Theorem~\ref{thm.liuwang}(a)
and again notice that $b(x) := x(p - q x)$
satisfies $b(x)\leq 0$ for all $x\leq 0$.

\item When $r_n=0$ and $s_n>0$,
using~\reff{eq.Xhat.recurrences}
and Lemma~\ref{lem.homogeneous.dy}, we obtain
\begin{eqnarray}
        \Xhat_{n+1}(x) 
        &=&
	\left[s_n \,+\, p\dfrac{\partial}{\partial x}
	\,+\, 
	q  \left(d - x\dfrac{\partial}{\partial x}\right)
	\right]
	\Xhat_{n}(x) 
	\nonumber\\
	&=&
	(s_n + qd)\Xhat_n(x) 
	\,+\,
	(p - q x) \Xhat_n'(x) 
\end{eqnarray}
Here we set $F(x) = \Xhat_{n+1}(x)$, $f(x) = \Xhat_n(x)$
and $g(x) = \Xhat_n'(x)$
in 
Theorem~\ref{thm.liuwang}(b).
We then notice $b(x) := (p - q x)$
satisfies $b(x)> 0$ for all $x\leq 0$.

\item When $r_n=s_n=0$,
using~\reff{eq.Xhat.recurrences}
and Lemma~\ref{lem.homogeneous.dy}, we obtain
\begin{eqnarray}
        \Xhat_{n+1}(x) 
        &=&
	\left[p\dfrac{\partial}{\partial x}
	\,+\, 
	q  \left(d - x\dfrac{\partial}{\partial x}\right)
	\right]
	\Xhat_{n}(x) 
	\nonumber\\
	&=&
	qd \Xhat_n(x) 
	\,+\,
	(p - q x) \Xhat_n'(x) 
\end{eqnarray}
Here again, we set $F(x) = \Xhat_{n+1}(x)$, $f(x) = \Xhat_n(x)$
and $g(x) = \Xhat_n'(x)$
in 
Theorem~\ref{thm.liuwang}(b).
We then notice $b(x) := (p - q x)$
satisfies $b(x)> 0$ for all $x\leq 0$.
\end{itemize}
In each of the above cases, we get that $\Xhat_{n+1}(x)$ is real-rooted.
\qed

As a consequence of Theorem~\ref{thm.xhat.realrooted},
the coefficients of the polynomial $\Xhat_n(x,1)$  must be log-concave and unimodal.

We end this section by remarking that a real-rootedness property associated to the Schett polynomials was conjectured by 
Lin and Ma \cite[Conjecture~5.1]{Lin_25};
we find this to be very intriguing.

\section*{Tool and computational resource disclosure}

The author declares that this manuscript was primarily written by him. 
However, ChatGPT was occasionally used to search for references, and also to review preliminary versions of the paper.
During one such review, ChatGPT suggested the content of Corollary~\ref{cor.reduced.combinter}.
All mathematical statements, proofs, references, and final wording were checked and approved by the author, who takes full responsibility for the content of the manuscript.

\section*{Acknowledgments}

We thank Alan Sokal and Andrea Sportiello for helpful discussions.
The author is currently supported by a Tsinghua University Shuimu scholarship.

\end{document}